\documentclass{amsart}

\usepackage[english]{babel}
\usepackage[utf8]{inputenc}
\usepackage{soul}
\usepackage{amscd}
\usepackage{latexsym}
\usepackage{amsfonts}
\usepackage{amssymb}
\usepackage{amsthm}
\usepackage{amsmath}

\usepackage{graphicx}
\usepackage{graphics}

\usepackage{multirow}
\usepackage{fancyvrb}
\usepackage{color}

\usepackage{enumerate}
\usepackage{verbatim}
\usepackage{import}

 \usepackage[colorlinks=true]{hyperref}
\hypersetup{urlcolor=blue, citecolor=red}
\usepackage{subfigure}

\usepackage{pgf,tikz}
\usepackage{mathrsfs}
\usetikzlibrary{arrows}

\usepackage{pgfplots}
\tikzset{ font={\fontsize{9pt}{12}\selectfont}}

\newtheorem{theorem}{Theorem}[section]
\newtheorem{proposition}[theorem]{Proposition}
\newtheorem{lemma}[theorem]{Lemma}

\newtheorem{corollary}[theorem]{Corollary}

\newtheorem*{theoremA}{Theorem A}
\newtheorem*{theoremnn}{Theorem}
\newtheorem*{corollaryn}{Corollary}
\theoremstyle{definition}
\newtheorem{definition}[theorem]{Definition}
\newtheorem{remark}[theorem]{Remark}

\newcommand{\newr}{}

\newcommand{\com}{\mathbb{C}}
\newcommand{\C}{\mathbb{C}}
\newcommand{\R}{\mathbb{R}}
\newcommand{\wcom}{\widehat{\mathbb{C}}}

\newcommand{\dis}{\mathbb{D}}
\newcommand{\cercle}{\mathbb{S}^1}
\newcommand{\re}{\rm{Re}}
\newcommand{\im}{\rm{Im}}

\def\cal {\mathcal}

\title[McMullen-like Julia in a Chebyshev-Halley Family]{Connected McMullen-like Julia sets in a Chebyshev-Halley Family}

\begin{author}[J.~Canela]{Jordi Canela}\thanks{The researcher J.~Canela was supported by the projects UJI-B2022-46 (Universitat Jaume I) and PID2020-118281GB-C32 (MCIU/AEI/FEDER/UE). The researcher A.~Garijo was  supported by the projects 
PID2020-118281GB-C33 (MCIU/AEI/FEDER/UE) and 2021SGR-633 (Generalitat de Catalunya). The research of P.Roesch was supported by the ANR LabEx CIMI (grant
ANR-11-LABX-0040) within the French State program “Investissements
d’Avenir”.}

\email{canela@uji.es}
\address{ %
Departament de Matem\`atiques\\
 Universitat Jaume I\\ 
 Av. Vicent Sos Baynat, s/n\\
  12071 Castelló de la Plana\\
  Spain.}
\end{author}

\begin{author}[A.~Garijo]{Antonio Garijo}
\email{antonio.garijo@urv.cat}
\address{ %
  Dept. d'Enginyeria Inform\`atica i Matem\`atiques\\
 Universitat Rovira i Virgili\\
 Av. Pa\"isos Catalans 26\\
 Tarragona 43007\\
 Spain.}
\end{author}

\begin{author}[P.~Roesch]{Pascale Roesch}
\email{pascale.roesch@math.univ-toulouse.fr}
\address{ %
  Institut de Math\'ematiques de Toulouse\\
 Universit\'e Paul Sabatier\\
  118, route de Narbonne \\
  31062 Toulouse Cedex \\
  France.}
\end{author}

\begin{document}

\begin{abstract}
In this paper we study a one parameter family of   rational maps  obtained by applying the Chebyshev-Halley  root-finding algorithms. 
We show that   the dynamics near  parameters  where the family presents some degeneracy might  be understood from the point of view of singular perturbations. More precisely,    we relate the dynamics of those maps   
with the one of the McMullen family $M_{\lambda}(z)=z^4 + \lambda /z^2$, using quasiconformal surgery. 


\end{abstract}

\maketitle

\section*{Introduction}

The root-finding algorithms are   widely known as iterative dynamical systems. On the chaotic part, the iterative method fails, so it is important to understand the chaotic set and how it varies with the method. 

The   root-finding algorithms applied to polynomials are  rational maps and, in this setting, there is a well defined dichotomy between the tame part, the Fatou set denoted   $\cal F(f)$,  and the chaotic part, the Julia set denoted $\cal J(f)$.  
On the Fatou set there is eventually  some limiting behaviour since   $\cal F(f)$   is  the set of points $z_0$ 
where the family of iterates $ (f^n)_{n \in \mathbb N }$ 
is normal if restricted to  some neighbourhood of $z_0$.

Among the rational cases, the quadratic polynomials---which are also the  simplest rational maps---have been well studied and proved to be  universal. This follows from the work of many authors including Douady, Hubbard \cite{DH1}, Lyubich \cite{Ly3}, and McMullen \cite{McMU}. Roughly speaking, it means that  the Julia sets  of those polynomials  appear in a lot of families : the dynamics can be restricted so as to look like quadratic.  It is much more easy to recognize the Julia set when it is connected, so the Mandelbrot set $M=\{c\in\C\mid \cal J(z^2+c)\hbox{ is connected}\}$ plays a fundamental role in parameter spaces. Its boundary $\partial M$ is the   bifurcation locus  of the quadratic family: the place where the dynamics changes drastically.   One of the first appearances of the universality of the quadratic family was observed by the presence of  ``copies" of the Mandelbrot set  in the family of  Newton's method applied to a  cubic polynomial  \cite{DH1}. 

However, there are rational maps whose Julia set is not homeomorphic to any quadratic Julia set. The family of rational maps $M_{n,d,\lambda}(z)= z^n + \lambda / z^d$ firstly introduced by C.~McMullen \cite{McM1}  is an example of this phenomenon.  Indeed, the Julia set of $ M_{n,d,\lambda}$ could be  a Cantor set of circles surrounding the origin or a Sierpinski carpet, among \newr{other possibilities} \cite{DLU}. These kind of Julia sets is not occurring for polynomials since, in this case, there is no Fatou component whose boundary is the whole  Julia set. The family of maps $M_{n,d,\lambda}$ is referred in the literature  as  {\it McMullen family}.

In this article we show that one can find   copies of the Julia set of maps in the McMullen family as subsets of  the Julia set of a family coming from Chebyshev-Halley root-finding algorithms.

\vskip 1em 

 The  family of {\it  Chebyshev-Halley   root-finding algorithms is} given by the recursive sequence $z_{n+1}=CH^f_\alpha(z_n) $,  where $\alpha$ is a complex parameter and     \begin{equation*}
CH^f_\alpha(z)=z-\left( 1+\frac{1}{2}\;\frac{L_{f}\left(  z\right) }
	{1-\alpha L_{f}\left(  z\right) }\right) \frac{f\left( z\right) }{f^{\prime
		}\left(  z\right) } \quad \quad \hbox{ with } \quad  L_{f}\left( z\right) =\frac{f\left( z\right) f^{\prime \prime }\left(
		z\right) }{\left( f^{\prime }\left( z\right) \right) ^{2}}.\end{equation*}

When this family is applied to the polynomial $f(z)=z^3-1$ and using the  new parameter  $ a=5-4\alpha$, one   gets the rational map 
\begin{equation*}
R_a(z)=\frac{2az^6+(15-a)z^3+3-a}{3z^2(5-a+(1+a)z^3)}.
\end{equation*} 

For $a=0$ the degree drops from $6$ to $5$. The goal of this work is to study this family  around the singular parameter $a=0$.
The main result  is the following. 

 \begin{theoremnn}   There exists a neighbourhood  $\Lambda$ of $0$ such that for  $a\in\Lambda \setminus \{0\}$   the map   $R^2_a$ is McMullen-like :     the map   $R^2_a$ is 	conjugated to a map in the McMullen family in some  annulus $A_a$.
 \end{theoremnn}

 Theorem A in the  \S\ref{sec:MCM} is  a more detailed statement  of this result (see also Theorem~\ref{thm:A}).  
 Moreover, in  \S \ref{sec:further_results} we prove that 
\begin{corollaryn}  For parameters $a\in\Lambda \setminus \{0\}$     the Julia set   $\mathcal J(R_a)$  contains the image by some homeomorphism of the Julia set of a map  in the  McMullen family. Moreover, the three different types of escaping  Julia sets of the  McMullen family appear in $\Lambda$.
 \end{corollaryn}


\vskip 1em
The paper  goes as follows. 
In \S  \ref{sec:CH} we give a short introduction \newr{to} the dynamics of the Chebyshev-Halley family. In \S \ref{sec:MCM} we \newr{state} the properties of the McMullen family, present Theorem A in \S \ref{sec:presentthma}, and list the properties to guarantee that a rational map of degree 6 is a McMullen map in \S\ref{sec:rigidity}.
In \S \ref{sec:R0} and \S \ref{sec:Dynamics:Ra} we discuss in detail the  dynamical properties of the rational maps $R_0$ and $R_a$, respectively. Then, \S \ref{sec:surgery} is mainly devoted to prove Theorem A, which is the content of Theorem~\ref{thm:A}. 
The proof is based on a cut and paste quasiconformal surgery procedure (see \cite{BF}) relating the dynamics of $R_a^2$ with the one of $M_{\lambda}$.  Finally,  in \S \ref{sec:further_results} we prove that the three types of Julia sets described in the Escape Trichotomy Theorem of \cite{DLU} can be found as  subset of  the Julia set of $R_a$ for different values of $a$. Furthermore, we remark that the same ideas work for the Chebyshev-Halley method applied to the polynomial $z^n-1$ with $n > 3$.

\section{Chebyshev-Halley family}\label{sec:CH}
   The  family of Chebyshev-Halley   root-finding algorithms is  given by the recursive sequence $z_{n+1}=CH^f_\alpha(z_n) $,
\begin{equation*}
	z_{n+1} =z_{n}-\left( 1+\frac{1}{2}\;\frac{L_{f}\left(  z_{n}\right) }
	{1-\alpha L_{f}\left(  \newr{z_{n}}\right) }\right) \frac{f\left(  z_{n}\right) }{f^{\prime
		}\left(  z_{n}\right) } \quad \quad \hbox{  with } \quad  L_{f}\left( z\right) =\frac{f\left( z\right) f^{\prime \prime }\left(
	z\right) }{\left( f^{\prime }\left( z\right) \right) ^{2}}\label{metodo}
\end{equation*}
 
\noindent and $\alpha\in\C$. \newr{This family  of root-finding algorithms was already studied  in \cite{Werner} (see also \cite{traub1964}); from the point of view of complex dynamics  it also appears in several works (see for instance \cite{ CTV, CCV2, Par})}.  In contrast to Newton's method, which has quadratic convergence for simple roots, these algorithms have cubic convergence, i.e.\ every simple root of a polynomial  is a super-attracting fixed point of local degree 3 of  $CH^f_\alpha$. 
 This family contains some well known root-finding algorithms. For example,  $\alpha=0$ corresponds to Chebyshev's method, $\alpha=1/2$ corresponds to Halley's method, and as $\alpha$ tends to infinity the family converges to Newton's method. 

In this paper, we   focus on the Chebyshev-Halley method applied to the polynomial $f(z)=z^3-1$.  It has  the expression

$$CH^f_{\alpha}(z)	=\frac{2-4\alpha+(-10-4\alpha)z^3+(-10+8\alpha )z^{6}}{6z^{2}(-2\alpha+(2\alpha -3)z^{3})}.$$ 
It can be simplified by considering the parameter $a=5-4\alpha$.  We obtain then  the following one parameter  family of rational maps defined on the Riemann sphere $\hat{\C}$ \newr{that we call}  $R_a$:

\begin{equation}\label{eq:Ra}
R_a(z)=\frac{2az^6+(15-a)z^3+3-a}{3z^2(5-a+(1+a)z^3)}  ,\quad a\in\C.
\end{equation}

\vskip 1em   For a parameter $a \notin \{0,3\}$, the rational map $R_a(z)$ exhibits $3$ free critical points. Let  $\zeta=e^{2\pi i/3}$. Given a choice of a punctual determination of a cubic  root,  the critical points  are
\begin{equation}\label{eq:crit}
	c_{a,j}=\zeta^j\sqrt[3]{\frac{15-8a+a^2}{a(a+1)}}, \;\;\;  \mbox{ where } j=0,1,2,
\end{equation}
and the critical values $v_{a,j}:=R_a(c_{a,j})$ are
\begin{equation}\label{eq:critvalue}
	v_{a,j}=\zeta^j\frac{(25-6a+a^2)}{(a-5)^2(a+1)}\sqrt[3]{a^2\frac{15-8a+a^2}{a+1}}.
\end{equation}

Some basic properties of $R_a$ are related to its symmetry.  It is straightforward to check that $R_a(\xi z) = \xi R_a(z)$ for any $\xi \in\mathbb U$, where  $\mathbb U=\{\xi\in\com \mid \xi^3=1 \}$ is the group of  third roots of the unity generated by  $\zeta=e^{2\pi i/3}$.  As a consequence, this symmetry provides a conjugacy in the dynamical plane. Note  that the orbits of the 3 free critical points are symmetric with respect to multiplication by a third root of the unity. We can conclude that the $a-$plane is the natural parameter plane of the family $R_a$. 

For the map  $R_a$,   the elements of $\mathbb U$ are  {\it super-attracting} fixed points  with local degree $3$  {(since the Chebyshev-Halley methods have order of convergence 3)}.  Hence,  to every $\xi \in \mathbb{U}$   is associated its basin of attraction
\begin{equation}\label{eq:basin}
	A_a(\xi)= \{ z \in \C \, | \, R_a^n(z) \to \xi  \, \hbox{ as } \, n \to \infty\}
\end{equation}
\noindent and its {\it immediate basin of attraction } $A_a^*(\xi)$   defined as \newr{the} connected component of $A_a(\xi)$ containing $\xi$. 

\vskip 1em The rational map $R_a$ has degree $ 6$, except for $a=0$ and $a=3$.  
The parameter $a=3$ corresponds to Halley's method, which is relatively simple to study since there are no critical points other than the super-attracting fixed points which correspond to the roots of the polynomial  (see  \cite{CCV2}).

\noindent At the parameter $a=0$ a singular perturbation happens. Indeed, if $a=0$ the dynamics at $\infty $ change drastically:  the map is  

\begin{equation}\label{eq:R0}
R_0(z)=\frac{15z^3+3}{3z^2(5+z^3)}
\end{equation} and has the points $\{0,\infty\}$  as a period two super-attracting cycle whereas for  $|a|\neq 0$  small enough   $\infty$ is a repelling fixed point.  
More precisely, the  point $0$ is in both cases  critical and is mapped with degree $2$ to $\infty$, but for $a =0$ the point $z=\infty$ is   sent  back with degree $2$ to $0$ (see Figure~\ref{fig:dynampert} (left)), while for  $a\neq0$ the point   $\infty$ becomes a fixed point of multiplier  $3(1+a)/2a$. Hence,  infinity is a repelling fixed point when  $a$ is close enough to $0$, $a\neq 0$.
\vskip 1em 
In summary, the dynamics of the map $R_0$ is completely understood. The map $R_0$ has degree $5$ and thus has $8$ critical points, which are the following.
 \begin{itemize}
 	\item \newr{The 3 solutions of the equation  $z^3-1=0$. They correspond to 3 superattracting fixed points of $R_0$ with local degree 3 and thus these 3 solutions are critical points of multiplicity 2 (\newr{so} 6 critical points counting multiplicity).}
 	\item \newr{The points $z=0$ and $z=\infty$ are simple critical points, which form the superattracting 2-cycle $\{0,\infty\}$.}
 \end{itemize}
 \newr{The case $a=3$ is also well understood: the Fatou set of $R_3$ consists of the basin of attraction of the third roots of the unity and there are no critical points other than the roots themselves.} 
 
 \newr{On the other hand, the map $R_a$, $a\notin\{0, 3\}$, has degree $6$ and thus has $10$ critical points, which are the following.}
 \begin{itemize}
 	\item \newr{As in the case of $R_0$, the 3 solutions of the equation  $z^3-1=0$ are critical points of multiplicity 2.}
 	\item \newr{The point $z=0$ is a simple critical point that is mapped onto $z=\infty$, which is now a fixed point.}
 	\item \newr{The 3 critical points $c_{a,j}$, $j=0,1,2$ (see \eqref{eq:crit}). These critical points converge to $z=\infty$ as $a$ tends to 0.}
 \end{itemize}
 Moreover, in the case of $R_a$,  $a\notin\{0, 3\}$, the dynamics of the 3 new critical points $c_{a,j}$, $j=0,1,2$, is tied by the symmetry of the family and, hence, the parameter space of $R_a$ is a one dimensional space parametrized by  $a\in \C$.  See Figures \ref{fig:param} and  \ref{fig:curveLambda}.

\newr{Furthermore, we can notice  similarities between  the  dynamical planes  of $R_0$ and part of that of $R_a$, namely on the closure of the whole basins of attraction under $R_0$ of the roots $\zeta^k$, $k=0,1,2$, denoted by $\overline {A_0(\zeta^k)}$ (see Figure  \ref{fig:dynampert}).}
This follows from the fact  $\overline {A_0(\zeta^k)}$ is compact (\newr{$A_0(\zeta^k)$ are in the complement of the basin at $\infty$})  and also that the map $R_a$ converges uniformly on compact sets of $\mathbb C$ to $R_0$ as $a$ tends to $0$.  Indeed,   it follows from the  expression  \eqref{eq:Ra} 
 of  $R_a$ that can be rewritten as
  
 $$
  R_a(z)= \frac{15 z^3 + 3 + a (z^3-1)(2z^3+1)}{ 3z^2(z^3+5) + a\cdot 3z^2(z^3-1)}.    $$
\vskip 1em

\begin{figure}[hbt!]
	\centering
	\subfigure{
		\includegraphics[width=150pt]{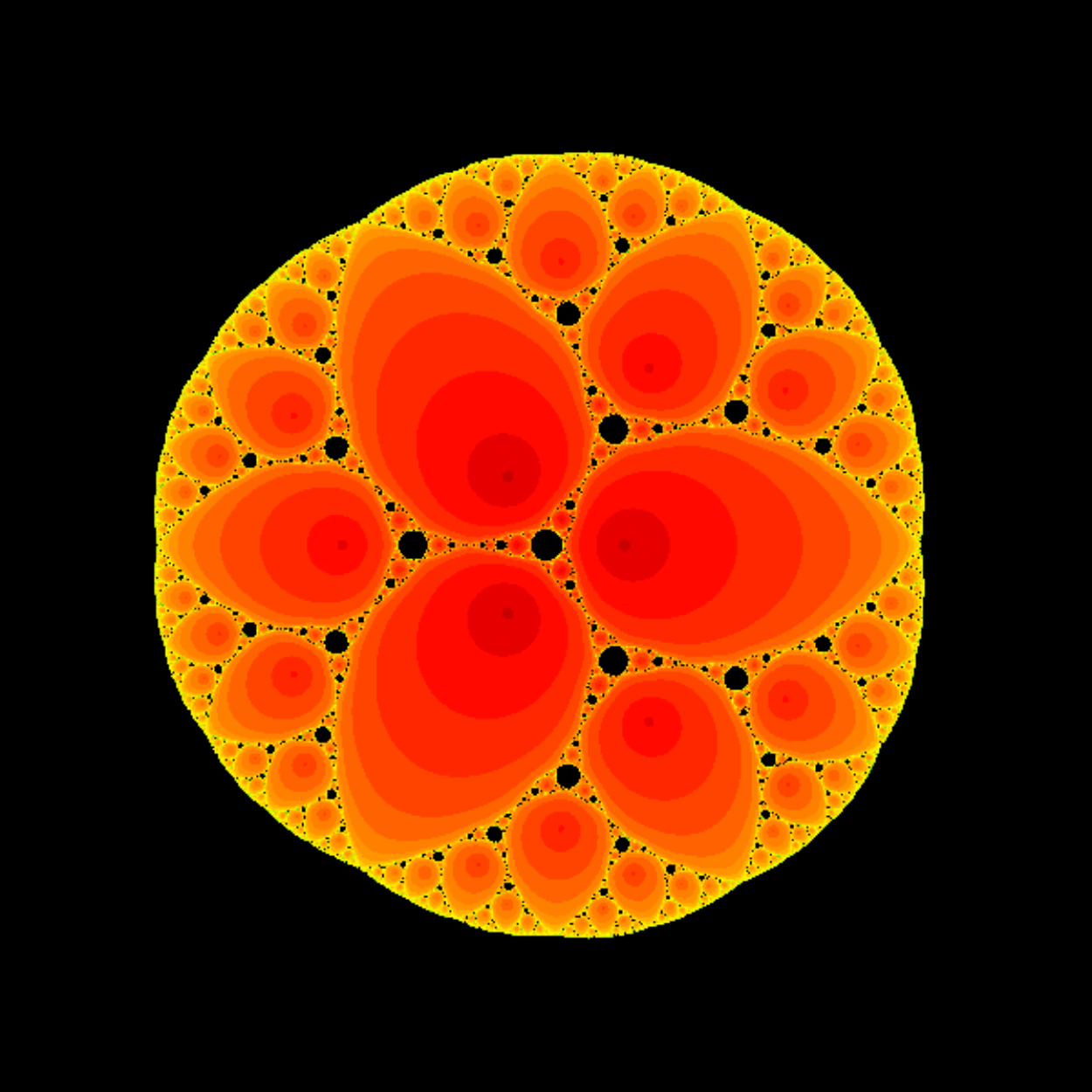}}
	\put(-65,73){{\footnotesize $\cdot 1$}}
	\put(-84.5,83){ {\tiny $\cdot$}} 
	\put(-89,86){ {\tiny $e^{\frac{2 \pi i}{3}}$}} 
	\put(-84.5,64.2){ {\tiny $\cdot$}} 
	\put(-95,55){ {\tiny $ e^{\frac{4 \pi i}{3}}$}} 
	\hspace{0.1in}
	\subfigure{
		\includegraphics[width=150pt]{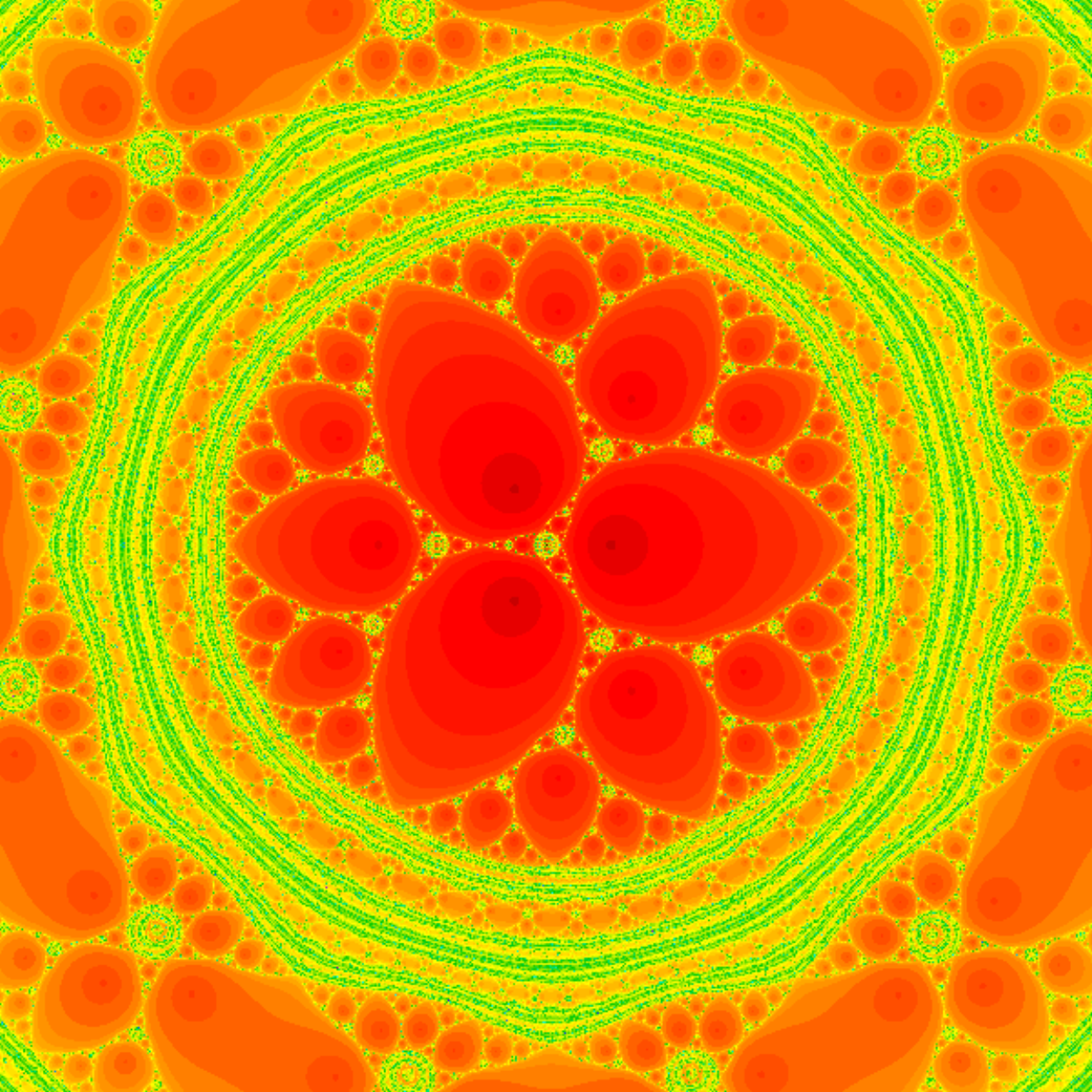}}
	\put(-70.5,73){ {\footnotesize $\cdot 1$}}
	\put(-84,81.2){ {\tiny $\cdot$}} 
	\put(-89,84){ {\tiny $e^{\frac{2 \pi i}{3}}$}} 
	\put(-84,66){ {\tiny $\cdot$}} 
	\put(-95,58){ {\tiny $ e^{\frac{4 \pi i}{3}}$}} 
	
	\caption{\small{Dynamical planes of $R_a(z)$ for $a=0$ (left) and $a=0.0001$ (right). In both dynamical planes red points represent points converging towards a third root of unity. For $a=0$ (left) black points represent points converging to the  super-attracting cycle $\{0,\infty\}$. We mark the third roots of unity $1, e^{2 \pi i/3}, e^{4 \pi i/3}$. } \label{fig:dynampert}}   
\end{figure}
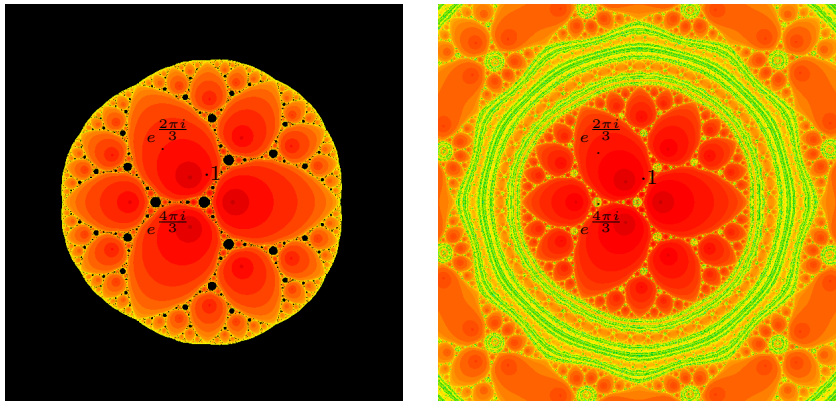

 In fact we can construct a holomorphic motion  of $\overline{A_0(\zeta^k)}$ (see Lemma \ref{lemma:motionbasins} and Figure \ref{fig:dynampert}). Since the concept of holomorphic motion \newr{will appear at several places} in this paper, we recall it in the next definition.
\begin{definition}\label{def:holomotion} A holomorphic motion of a set $X\subset \widehat \C$ parameterized by a domain $\Lambda\subset \C$  and based at $\lambda_0\in\Lambda$ is a map $H:\Lambda\times X\to\widehat \C$ such that 
\begin{itemize}
\item $H(\lambda_0,z)=z \quad \forall z\in X$
\item $z\mapsto H(\lambda, z) $ is injective  for each $\lambda \in\Lambda$
\item the map $\lambda\mapsto H(\lambda ,z)$ is holomorphic for each $z\in X$. 
\end{itemize}
\end{definition}




  \section{McMullen family and main result}\label{sec:MCM}
As mentioned before, \newr{the} singular perturbations \newr{$M_{n,d,\lambda}(z)= z^n + \lambda / z^d$} where introduced by C.~McMullen~\cite{McM1}  in order to prove the existence of buried Julia components, i.e.\ connected components of the Julia set which do not intersect the boundary of any Fatou component. More specifically, he provided the first example of rational map whose Julia set  is a Cantor set of quasicircles. Afterwards, R.~Devaney, D.~Look, and D.~Uminsky \cite{DLU} provided the following classification for the Julia set of McMullen maps when the orbit of all critical points  tend to $\infty$ (see Figure~\ref{fig:dyn_plane}).

\begin{theoremnn}[Escape Trichotomy, \cite{DLU}]
	Assume that all critical points of $M_{n,d,\lambda}$ belong to $A_{M_{n,d,\lambda}}(\infty)$. Then, exactly one of the following occurs.
	\begin{itemize}
		\item All critical points of $M_{n,d,\lambda}$ belong to $A_{M_{n,d,\lambda}}^*(\infty)$. Then, the Julia set $M_{n,d,\lambda}$ is a Cantor set of points.
		\item All critical points of $M_{n,d,\lambda}$ are mapped in exactly two iterates into  $A_{M_{n,d,\lambda}}^*(\infty)$. Then, the Julia set $\mathcal{J}(M_{n,d,\lambda})$ is a Cantor set of circles.
		\item All critical points of $M_{n,d,\lambda}$ are mapped in exactly $m>2$ iterates into  $A_{M_{n,d,\lambda}}^*(\infty)$. Then, the Julia set $\mathcal{J}(M_{n,d,\lambda})$ is a Sierpinsky carpet.
	\end{itemize}
\end{theoremnn}

\subsection{The main result}\label{sec:presentthma}

 The goal of this paper is to relate the dynamics of $R_a$ with the dynamics of the \newr{following} McMullen map
\begin{equation}\label{eq:Mcmullen}
M_{\lambda}(z):=M_{4,2,\lambda}(z)=z^4+\frac{\lambda}{z^2}
\end{equation}
for parameters $a$ close to 0. In  Figure~\ref{fig:dyn_plane} we can observe that there appear structures in $\mathcal{J}(R_a)$ similar to  the Cantor sets of circles,  the Sierpinski carpet and the Cantor set of the family of maps $M_{\lambda}$. In Figure \ref{fig:param} we compare the parameter plane of $R_a$ near the origin \newr{to} the parameter plane of $M_{\lambda}$.

\begin{figure}[hbt!]
    \centering
   \subfigure{
     \includegraphics[width=170pt]{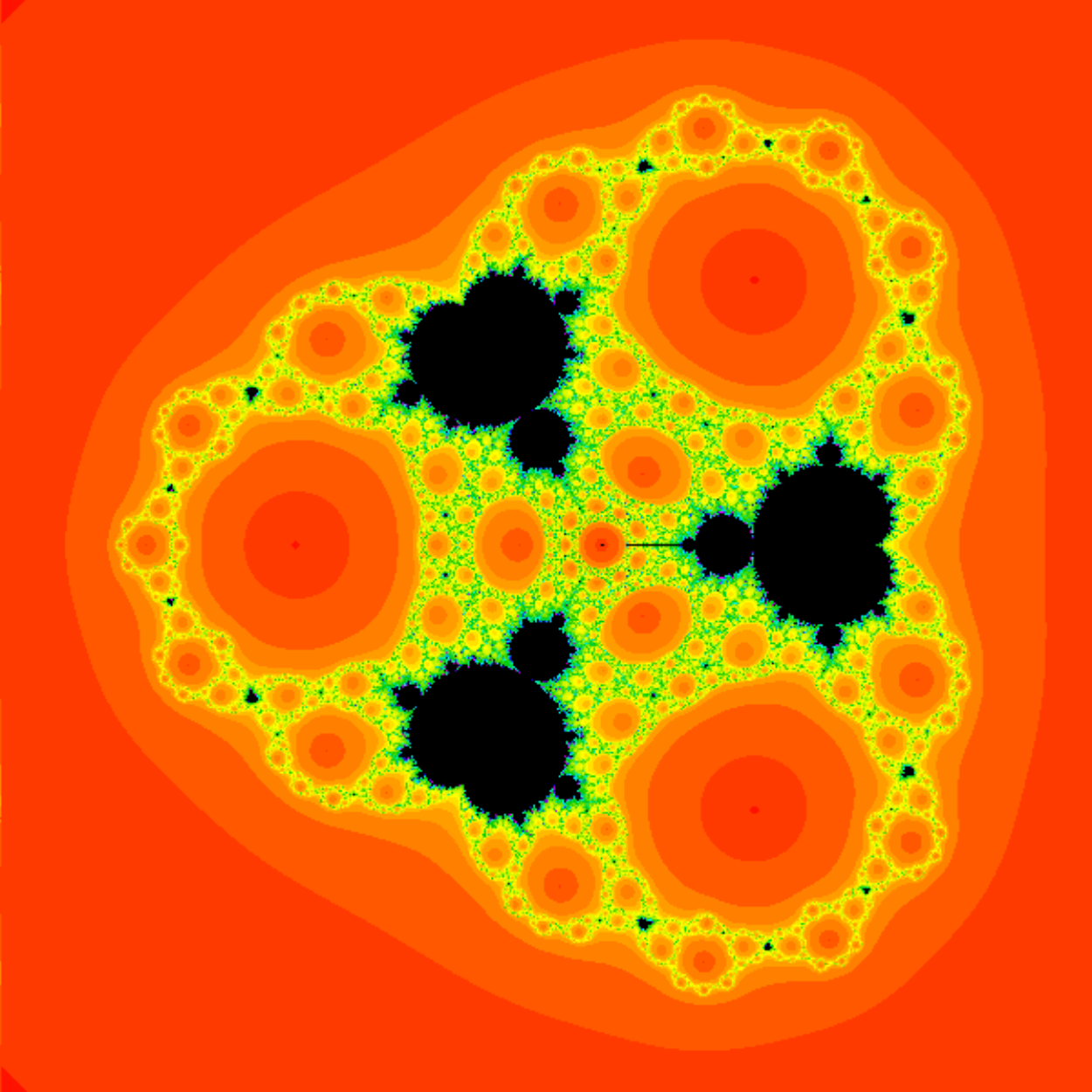}}
    \hspace{0.1in}
    \subfigure{
  \includegraphics[width=170pt]{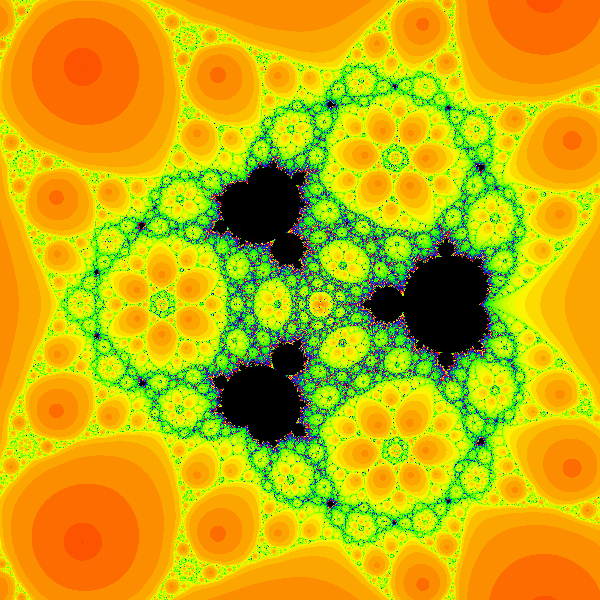}}
  \caption{\small{ \newr{On} the left side we \newr{plot} the parameter plane of  $M_{\lambda}$ \eqref{eq:Mcmullen}  and \newr{on} the right side the parameter plane of  $R_a$ \eqref{eq:Ra} near $a=0$.} \label{fig:param}}   
  \end{figure}

The main result of this work is the following theorem, which  basically states that for $a$ in a neighbourhood  of the origin (see Section \ref{sec:Dynamics:Ra}   for details) the second iterate of $R_a$    \eqref{eq:Ra} is conjugate to    some  $M_{\lambda}$  \eqref{eq:Mcmullen} in a concrete annulus defined in the dynamical plane.    The result implies that a copy of the Julia set of $M_{\lambda}$ is contained in the Julia set of $R_a$.

\begin{theoremA}\label{ThmA}   There exists a neighbourhood  $\Lambda$ of $0$ such that for  $a\in\Lambda \setminus \{0\}$   the map   $R^2_a$ is 	conjugated to a map in the McMullen family $M_{\lambda}$. More precisely,  there exists a quasiconformal map $\varphi_a:\wcom\rightarrow\wcom$ such that $\varphi_a\left(R^2_a(z)\right)=M_{\lambda(a)}\left(\varphi_a(z)\right)$ for all $z$ in some  annulus $A_a$ 
with $\varphi_a^{-1}(\mathcal{J}(M_{\lambda(a)}))\subset A_a$,  where $a \mapsto \lambda(a)$ is a map defined on $\Lambda$.
\end{theoremA}

\begin{figure}[p]
	\centering
	\subfigure{ \includegraphics[width=140pt]{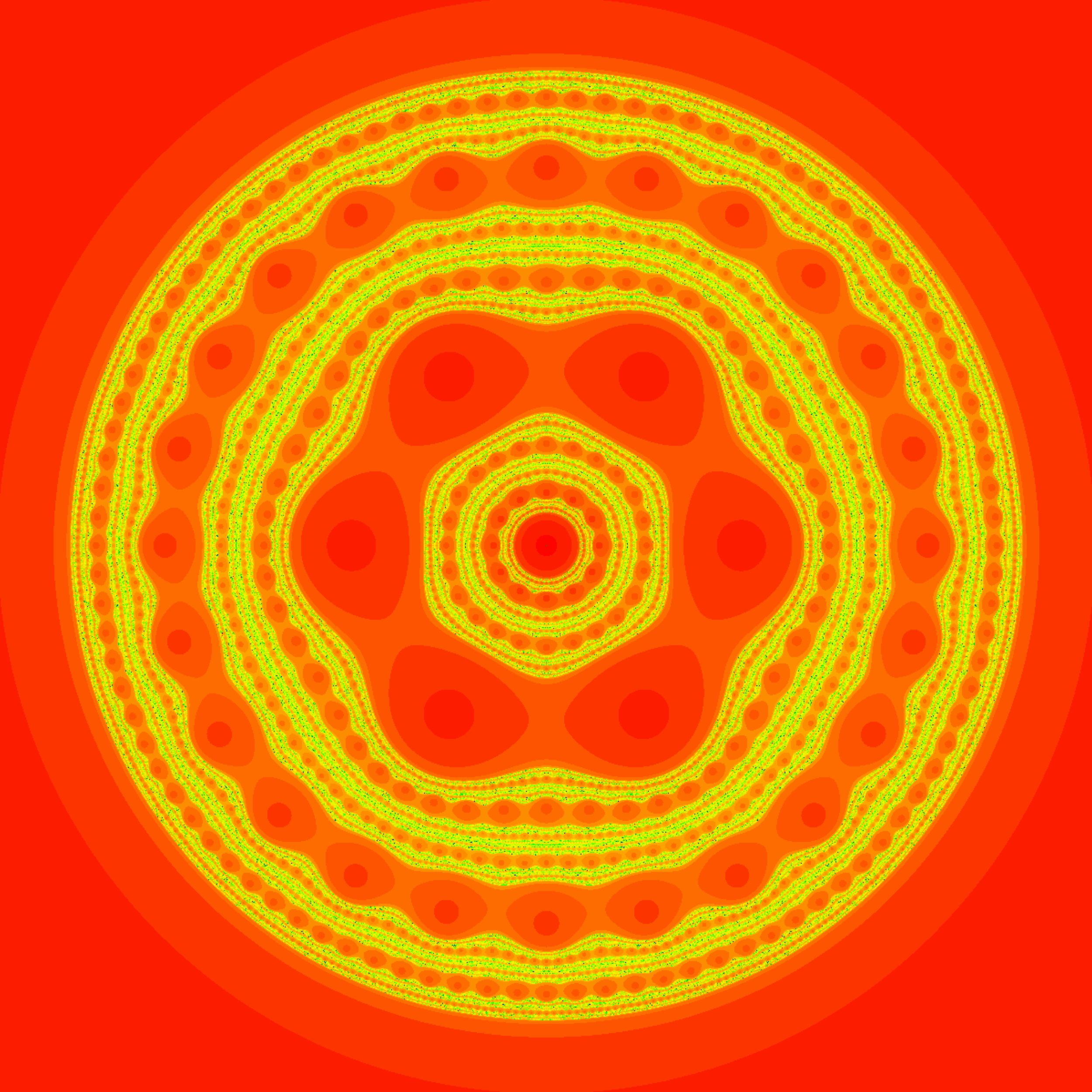}}
	\hspace{0.01in}
	\subfigure{ \includegraphics[width=140pt]{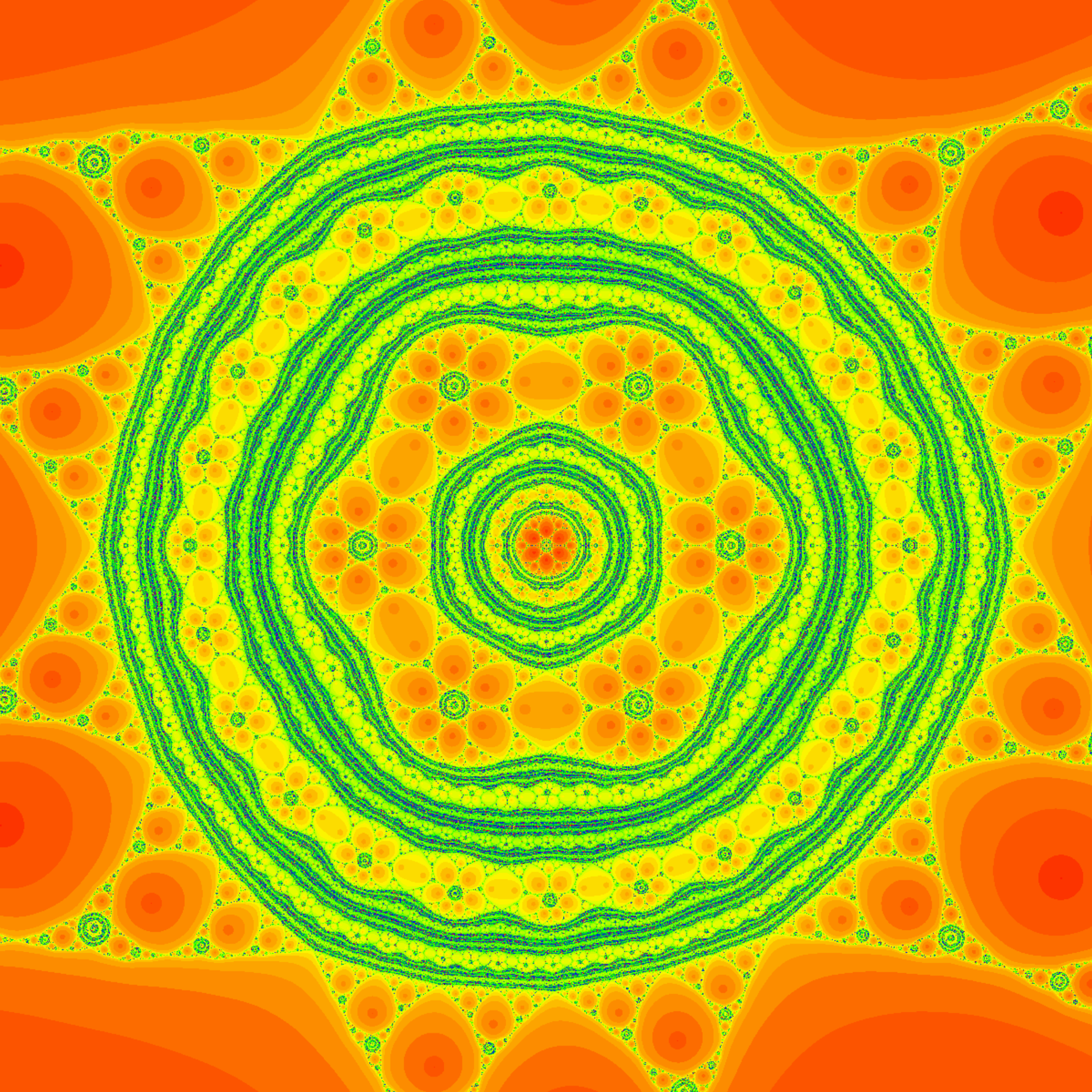}}
	\subfigure{\includegraphics[width=140pt]{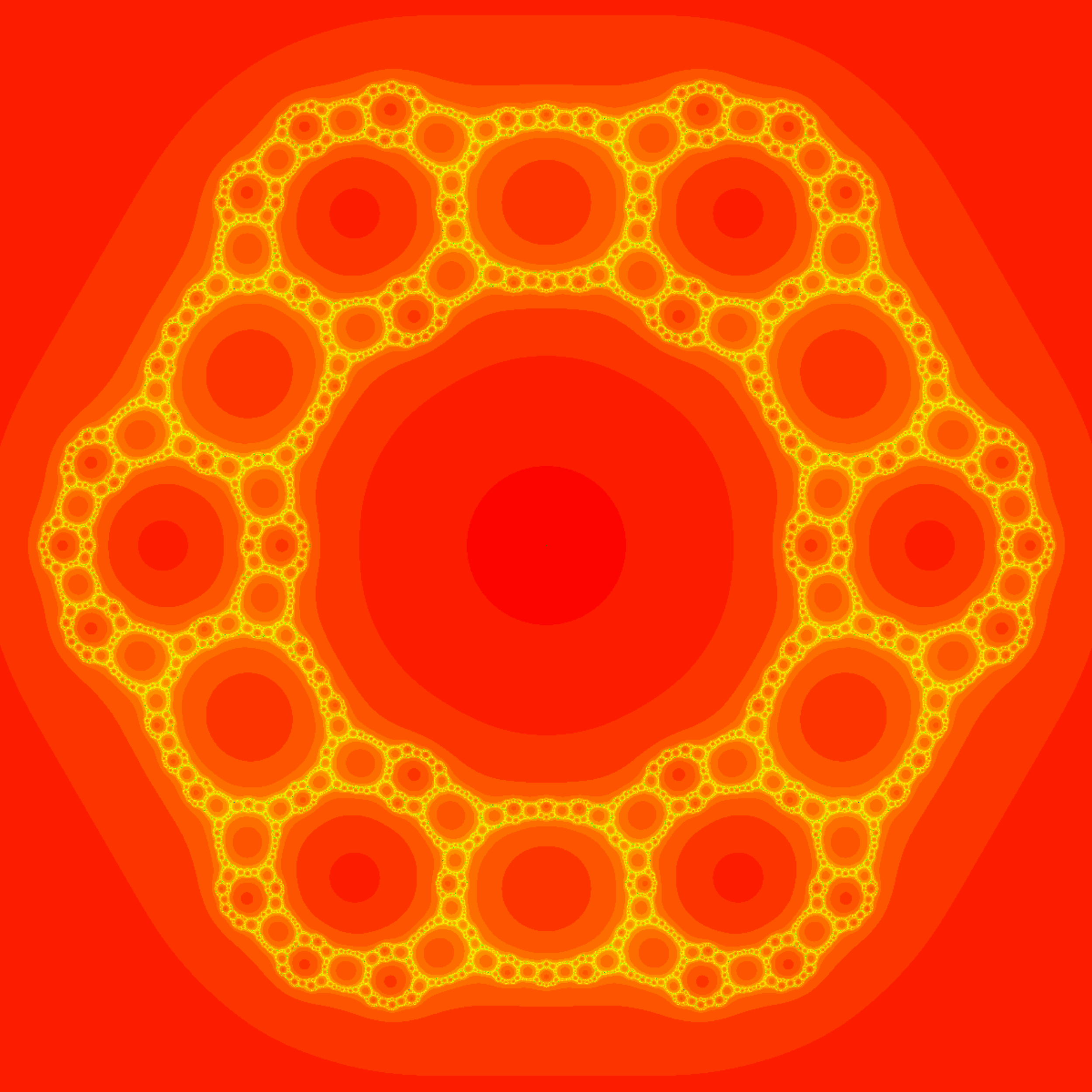}}
	\hspace{0.01in}
	\subfigure{\includegraphics[width=140pt]{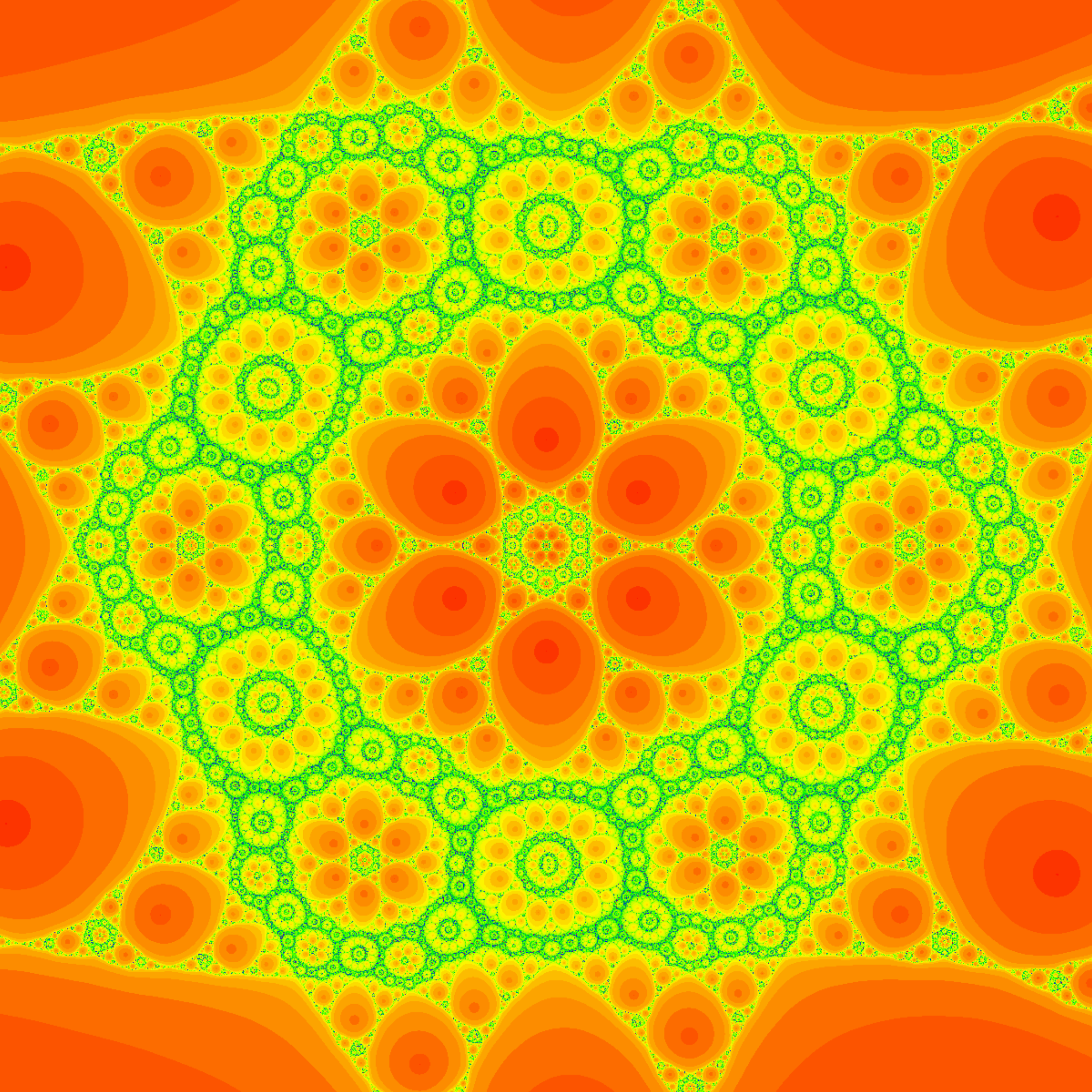}}
	\subfigure{\includegraphics[width=140pt]{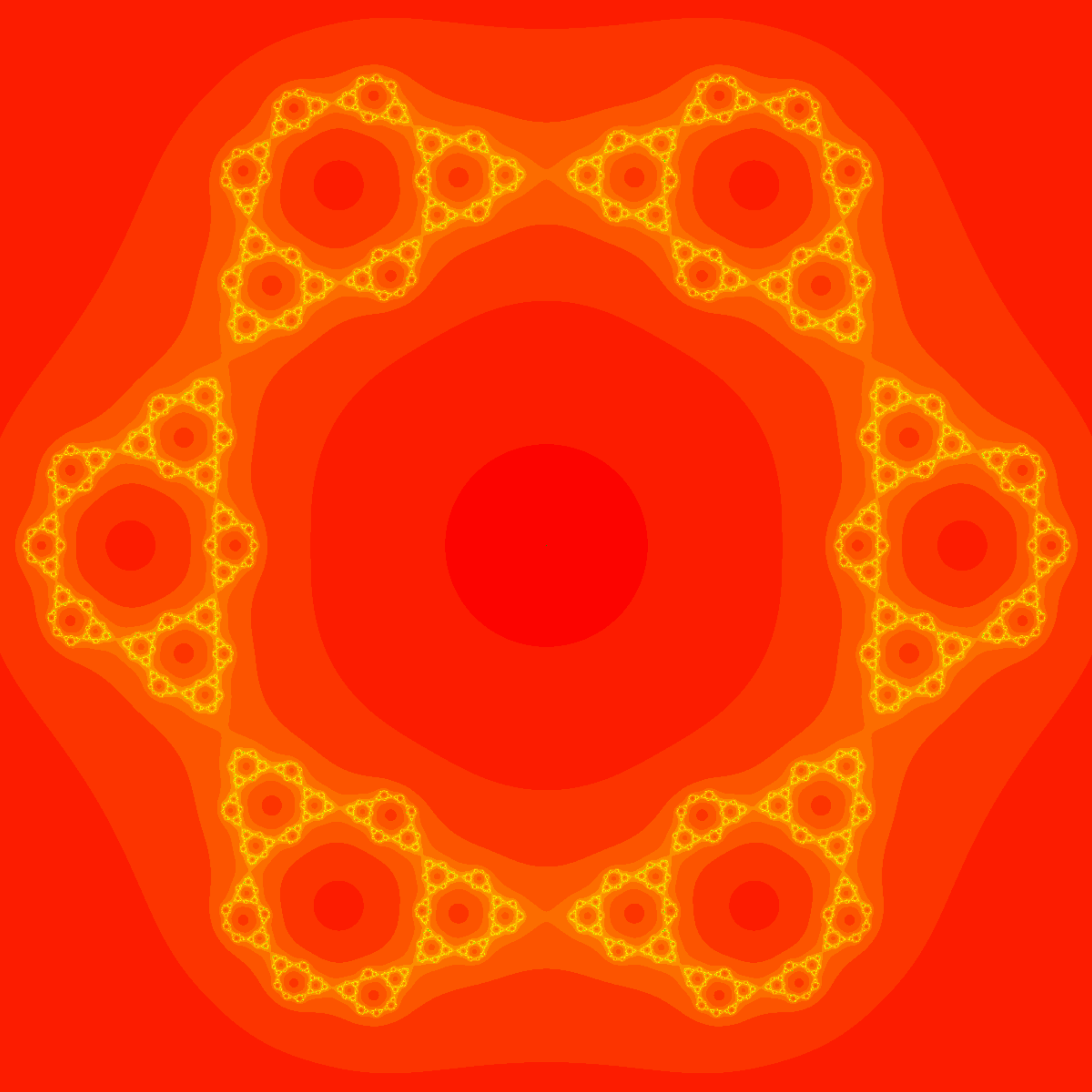}}
	\hspace{0.01in}
	\subfigure{\includegraphics[width=140pt]{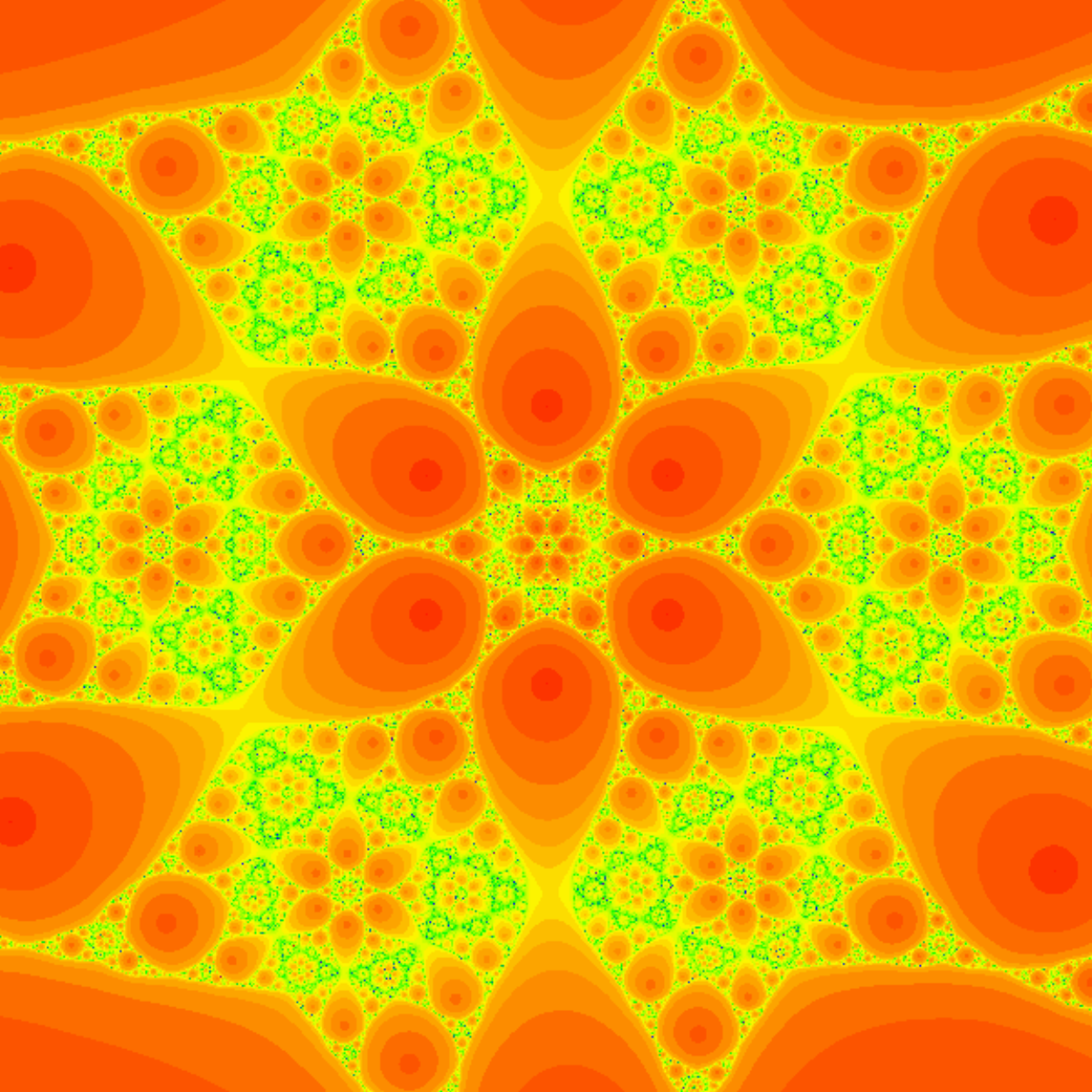}}
	
	\caption{\small{In the first column we  \newr{present}  the dynamical plane of $M_{\lambda}$ \eqref{eq:Mcmullen} for $\lambda=0.005$ (top)  $\lambda=-0.28$ (middle) and $\lambda=-0.455$ (bottom). In the second column  the dynamical plane of $R_a$ (see \eqref{eq:Ra}) for $a=-0.0003$ (top)  $a=-0.0164$ (middle) and $a=-0.028$ (bottom) with $\re(z)\in (-0.245, 0.245)$ and $\im(z)\in (-0.245, 0.245)$. In the dynamical plane of $ R_a$ red points represent points converging to a third root of unity while in the McMullen family $M_{\lambda}$  red points represent points converging to infinity. \label{fig:dyn_plane}}}   
\end{figure}

\subsection{Rigidity of McMullen's family}\label{sec:rigidity}\ 

 The following proposition is  a  characterisation for a  rational map to be  linearly conjugate to McMullen's map $M_{\lambda}(z)=z^4+\frac{\lambda}{z^2}$ \eqref{eq:Mcmullen}.

\begin{proposition}\label{prop:rigidity}
 Let $Q:\wcom\rightarrow \wcom$ be a degree 6 rational map satisfying the following properties:
 \begin{enumerate}[a)]
 	\item The point $z=\infty$ is super-attracting of local degree 4;
 	\item The point $z=0$ is a double preimage of $z=\infty$.
 	\item The map $Q$ is symmetric with respect to multiplication by a third root of the unity, i.e. $Q(\xi z)=\xi Q(z)$ where   $\xi^3=1$. 
 	\item The map $Q$ has exactly 6 different simple critical points (other than $z=0$ and $z=\infty$) which are mapped under $Q$ onto exactly 3 different critical values.
 \end{enumerate}
Then $Q$  is linearly conjugate to 
$$M_\lambda (z)=z^4+\frac{\lambda}{z^2}.$$ 
\end{proposition}
\proof

 It follows directly from a), b) and c) that the map $Q$ can be written as
$$Q(z)=\frac{az^6 +\tilde{b}z^3+\tilde{\lambda}}{z^2},$$
where $a,\tilde{b},\tilde{\lambda}\in\com$ and $a \tilde{\lambda}\neq 0$. This map is linearly conjugate to
$$M_{b,\lambda}(z)=\frac{z^6 +bz^3+\lambda}{z^2},$$
where $b,\lambda\in\com$, by the map $z\rightarrow z/\sqrt[3]{a}$. We now use property d) to prove that $b=0$. The critical points of $M_{b,\lambda}$ are solutions of
$$4z^6+bz^3-2\lambda=0.$$
Writing $w=z^3$, we get the equation $4w^2+bw-2\lambda=0$, whose solutions are
$$w_\pm=\frac{-b\pm\sqrt{b^2+32\lambda}}{8}.$$
 By d)  $M_{b,\lambda}$ has 6 different simple critical points (other than $z=0$ and $z=\infty$), we conclude that $\Delta:=\sqrt{b^2+32\lambda}\neq 0$. Moreover, we have that $w_+\cdot w_-=-\lambda/2$.

Let $\zeta=e^{2\pi i/3}$. These 6 critical points of $M_{b,\lambda}$ are labelled:
\begin{itemize}
	\item $c_{1,+}=\sqrt[3]{w_+}$, $c_{2,+}=\zeta\cdot c_{1,+}$, $c_{3,+}=\zeta^2\cdot c_{1,+}$;
	\item $c_{1,-}=\sqrt[3]{w_-}$, $c_{2,-}=\zeta\cdot c_{1,-}$, $c_{3,-}=\zeta^2\cdot c_{1,-}$.
\end{itemize}
Notice that the labelling of the critical points depends on arbitrary choices of the square and cubic roots, but this is not important since we find them all. By hypothesis,  $M_{b,\lambda}$ maps these 6 critical points to exactly three different critical values. Notice that if $M_{b,\lambda}(c_{j,+})=M_{b,\lambda}(c_{k,+})$ with $j\neq k$ then, by symmetry, $M_{b,\lambda}(c_{1,+})=M_{b,\lambda}(c_{2,+})=M_{b,\lambda}(c_{3,+})$. It would then follow that $M_{b,\lambda}$ maps the six critical points to two critical values, which is not possible by hypothesis. We can conclude that $M_{b,\lambda}(c_{j,+})=M_{b,\lambda}(c_{k,-})$ for some $j$ and $k$. For all critical points $c_{j,\pm}$ we have
$$M_{b,\lambda}(c_{j,\pm})=c_{j,\pm}^4+b\cdot c_{j,\pm}+\frac{\lambda}{c_{j,\pm}^2}=c_{j,\pm}\left(c_{j,\pm}^3+b+\frac{\lambda}{c_{j,\pm}^3}\right)=c_{j,\pm}\left(w_\pm+b+\frac{\lambda}{w_\pm}\right)=$$
$$=c_{j,\pm}\left(w_\pm+b-2w_\mp\right).$$
By taking the third power on both sides of the equality $M_{b,\lambda}(c_{j,+})=M_{b,\lambda}(c_{k,-})$ we obtain, independently of $j$ and $k$, the equality:
$$w_+\left(w_++b-2w_-\right)^3=w_-\left(w_-+b-2w_+\right)^3.$$
By using that $w_\pm=(-b\pm\Delta)/8$, where $\Delta:=\sqrt{b^2+32\lambda}\neq 0$, simple computations yield that the previous equation is equivalent to $432\Delta^3b=0$. Since $\Delta\neq 0$, we conclude that $b=0$, which finishes the proof.
\endproof

\section{Dynamics of the map $R_0$}\label{sec:R0}

In this section we study the dynamics of the unperturbed map $R_0$. We start by analyzing the relation of the repelling fixed points of $R_0$ and the basins of attraction $A_0(\zeta^j)$, where $j=0,1,2$ (see \eqref{eq:basin}). The maps $R_a(z)$ have 3 fixed points in $\com$ other \newr{than} the third roots of the unity, given by 
\begin{equation}\label{eq:fixed}
x_{a,j}=\zeta^j\sqrt[3]{\frac{a-3}{a+3}}, \;\;\; \mbox{ where } \zeta=e^{2\pi i/3} \mbox{ and } j=0,1,2.
\end{equation}

 In this case we take the determinacy of the third root such that $\sqrt[3]{-1}=-1$ (which  is well defined for $|a|$ small). Notice that for $a=0$ these three fixed points are given by $x_{0,j}=-\zeta^j$. \newr{Recall} that the third roots of the unity are $\zeta^j$ with $j=0,1,2$.

\begin{lemma}\label{lem:fixed0}
Each repelling  fixed point $x_{0,j}$ of $R_0$  belongs to the boundary of the immediate basins of attraction  $A_a^*(\zeta^{k})$ and $A_a^*(\zeta^{k'})$   where  $j,k,k' \in  \{0,1,2\}$ and $k'\neq j \neq k \neq k'$.
\end{lemma}

\begin{proof}
If $a=0$ the critical points are the roots of the unity and the points $0$ and $\infty$, which form a super-attracting 2-cycle. Every root of the unity $\zeta^j$ is a super-attracting fixed point of local degree 3. Since its immediate basin of attraction $ A_0^*(\zeta^j)$  cannot contain any free critical point, the Böttcher coordinate extends until reaching the boundary of the immediate basin of attraction. It follows that the fixed dynamical rays, of angles 0 and 1/2, land at $\partial A_{0}^*(\zeta^j)$. They either land at two different fixed points or at a common fixed point. 

If they land at different fixed points we are done. Indeed, since there are only 3 fixed points other than the roots of the unity, it follows from the symmetry in the dynamical plane  that every $x_{0,j}$ belongs to the boundary of exactly 2 immediate basins of attraction.

To finish the proof we have to see that these fixed rays cannot land at a common fixed point. We focus on the basin of attraction of the root $\zeta^0=1$. Since $R_0$ leaves the real line invariant, the map $i(z)=\overline {z}$ conjugates $R_0$ with itself.  We can conclude that if a fixed ray lands at $x_{0,j}\in\partial A^*_{0}(1)$, then the other fixed ray lands at  $\overline{x_{0,j}}\in\partial A^*_{0}(1)$. Since we are assuming that they land at the same point, we can deduce that they both land at $x_{0,0}=-1$. Using the symmetry with respect to rotation by a third root of the unity, we can conclude that $-1\in\partial A^*_{0}(1)$, that $-\zeta\in\partial A^*_{0}(\zeta)$, and that  $-\zeta^2\in\partial A^*_{0}(\zeta^2)$. However, this is impossible since then the 3 different basins of attraction would have a non-empty intersection. It also follows from the previous argument that  $x_{0,0}\notin\partial A^*_{0}(\zeta^0)$. By symmetry we conclude that $x_{0,j}\notin\partial A^*_{0}(\zeta^j)$, $j=0,1,2$.
\end{proof}

We will construct a partition of the dynamical plane of $R_0$ using dynamical rays. 
The third roots of the unity are  super-attracting fixed points of local degree 3 under the map $R_0$. Since $R_0$ has no free critical points, the B\"ottcher coordinate extends to the whole immediate basin of attraction of each third root of the unity $\zeta^j$, $j=0,1,2$. It follows that each $A^*_{0}(\zeta^j)$ contains two invariant rays, of angles 0 and $1/2$, which land at two different fixed points on $\partial A^*_{0}(\zeta^j)$  by Lemma~\ref{lem:fixed0}. 

Let $\eta$ denote the curve obtained \newr{as} the union of these 6 fixed rays (2 fixed rays for each one of the third root of the unity). It follows from Lemma~\ref{lem:fixed0}   that $\eta$ is a simple closed curve which surrounds $z=0$ and is invariant under rotation by third roots of the unity (see Figure \ref{esquemarays0}). Moreover, for each the root $\zeta^j$, the curve $\eta$ separates it from its   preimages. Indeed, the root 1 has two different real preimages, one in the interval  $(-1,0)$ and another one in the interval $(-\infty,-1)$. Notice that $\eta$ contains the repelling fixed point $-1$. 

\begin{figure}[hbt!]
\label{fig:conf0}
\centering
  \setlength{\unitlength}{250pt}%
 \begin{picture}(1,1)%
    \put(0,0){\includegraphics[width=\unitlength,page=1]{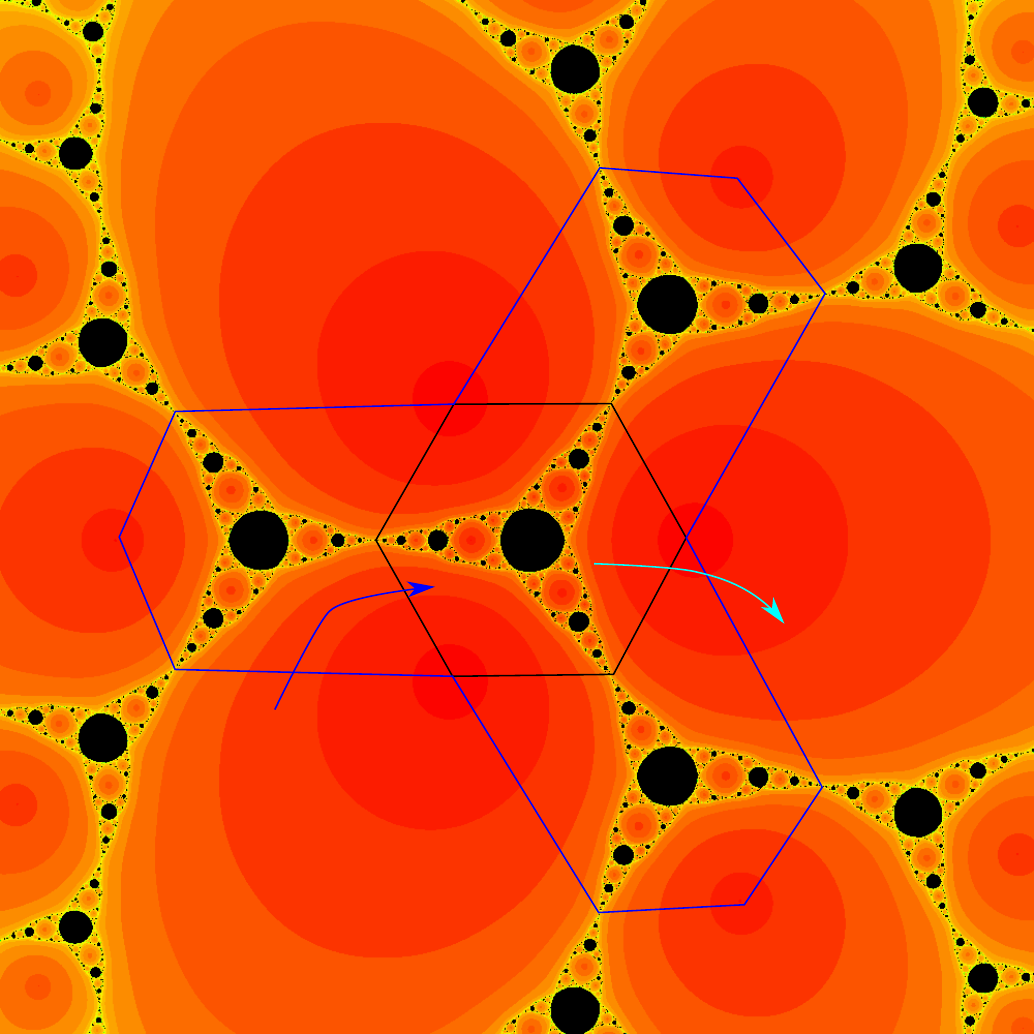}}%
    \put(0.29,0.37){\color[rgb]{0,0,0.99607843}\makebox(0,0)[lt]{\smash{\begin{tabular}[t]{l}2:1\end{tabular}}}}%
    \put(0.69,0.44){\color[rgb]{0,1,1}\makebox(0,0)[lt]{\smash{\begin{tabular}[t]{l}2:1\end{tabular}}}}%
    \put(0.50272286,0.62347955){\color[rgb]{0,0,0}\makebox(0,0)[lt]{\smash{\begin{tabular}[t]{l}0\end{tabular}}}}%
    \put(0.34,0.54326728){\color[rgb]{0,0,0}\makebox(0,0)[lt]{\smash{\begin{tabular}[t]{l}$\frac{1}{2}$\end{tabular}}}}%
    \put(0.47,0.75456876){\color[rgb]{0,0,1}\makebox(0,0)[lt]{\smash{\begin{tabular}[t]{l}$\frac{1}{6}$\end{tabular}}}}%
    \put(0.23460743,0.63307414){\color[rgb]{0,0,1}\makebox(0,0)[lt]{\smash{\begin{tabular}[t]{l}$\frac{1}{3}$\end{tabular}}}}%
    \put(0,0){\includegraphics[width=\unitlength,page=2]{conf0.pdf}}%
    \put(0.43305198,0.522){\color[rgb]{0,1,1}\makebox(0,0)[lt]{\smash{\begin{tabular}[t]{l}$\frac{2}{3}$\end{tabular}}}}%
    \put(0.48,0.58){\color[rgb]{0,1,1}\makebox(0,0)[lt]{\smash{\begin{tabular}[t]{l}$\frac{5}{6}$\end{tabular}}}}%
    \put(0.48,0.46596651){\color[rgb]{0,1,1}\makebox(0,0)[lt]{\smash{\begin{tabular}[t]{l}$Int(\eta_0)$\end{tabular}}}}%
    \put(0.57,0.52){\color[rgb]{0,1,1}\makebox(0,0)[lt]{\smash{\begin{tabular}[t]{l}$\eta_0$\end{tabular}}}}%
    \put(0.23,0.26677645){\color[rgb]{0,0,1}\makebox(0,0)[lt]{\smash{\begin{tabular}[t]{l}$Ext(\eta_{\infty})$\end{tabular}}}}%
    \put(0.05,0.45){\color[rgb]{0,0,1}\makebox(0,0)[lt]{\smash{\begin{tabular}[t]{l}$\eta_{\infty}$\end{tabular}}}}%
    \put(0.48,0.36){\color[rgb]{0,0,0}\makebox(0,0)[lt]{\smash{\begin{tabular}[t]{l}$Int(\eta)$\end{tabular}}}}%
    \put(0.61,0.39){\color[rgb]{0,0,0}\makebox(0,0)[lt]{\smash{\begin{tabular}[t]{l}$\eta$\end{tabular}}}}%
    
  \end{picture}%
\caption{\small Dynamical plane of $R_0$ and sketch of the dynamical rays involved in the construction}.
\label{esquemarays0}
\end{figure}

The preimage of $\eta$ under $R_0$ consists of the union of 3 different simple closed curves which intersect at the roots $\zeta^j$ (see Figure~\ref{esquemarays0}). One of the curves is $\eta$, which is mapped with degree 1 onto itself. The other curves are $\eta_{\infty}\subset \overline{ Ext(\eta)} $ and $\eta_0\subset  \overline{ Int(\eta)}  $. By relabelling the external rays if necessary (for each $A_{0}^*(\zeta^j)$ we can choose which fixed ray is 0 and which is 1/2), we can assume that $\eta_{\infty}$ contains the rays of angles $1/6$ and $1/3$ and that $\eta_{0}$ contains the rays of angles $2/3$ and $5/6$. Besides those rays, the curve $\eta_{\infty}$ and $\eta_0$ contain the preimages of the rays of angles 0 and 1/2 attached at the preimages of the roots of the unity contained in $Ext(\eta)$ and $Int(\eta)$, respectively (see Figure~\ref{esquemarays0}). It follows that $\eta_{\infty}$ and $\eta_0$ are simple closed curves that surround $z=0$ and are invariant under rotation by a third root of the unity.

\begin{lemma}\label{lem:confD} $R_0: Ext(\eta_{\infty}) \rightarrow Int(\eta)$ and $R_0:Int(\eta_0)\rightarrow Ext(\eta)$ are proper maps  of degree 2.
\end{lemma}
\proof 
 The open  set  $Ext(\eta_{\infty})$ contains no other preimage of $z=0$  than $z=\infty$. Indeed, $z=0$  has three preimages under $R_0$  other than $z=\infty$, which are given by $-\zeta^j\sqrt[3]{1/5}$   for $j=0,1,2$ and are contained in $Int(\eta)$. It follows that  $R_0:Ext(\eta_{\infty}) \rightarrow Int(\eta)$ is  a proper map. The degree of this proper map is 2 since $z=\infty$ is mapped 2 to 1 onto $z=0$ under $R_0$.  Analogously, $R_0: Int(\eta_0)\rightarrow Ext(\eta)$ is a proper map of degree 2.
\endproof
\begin{figure}[hbt!]
\centering
  \setlength{\unitlength}{250pt}%
 \begin{picture}(1,1)%
    \put(0,0){\includegraphics[width=\unitlength,page=1]{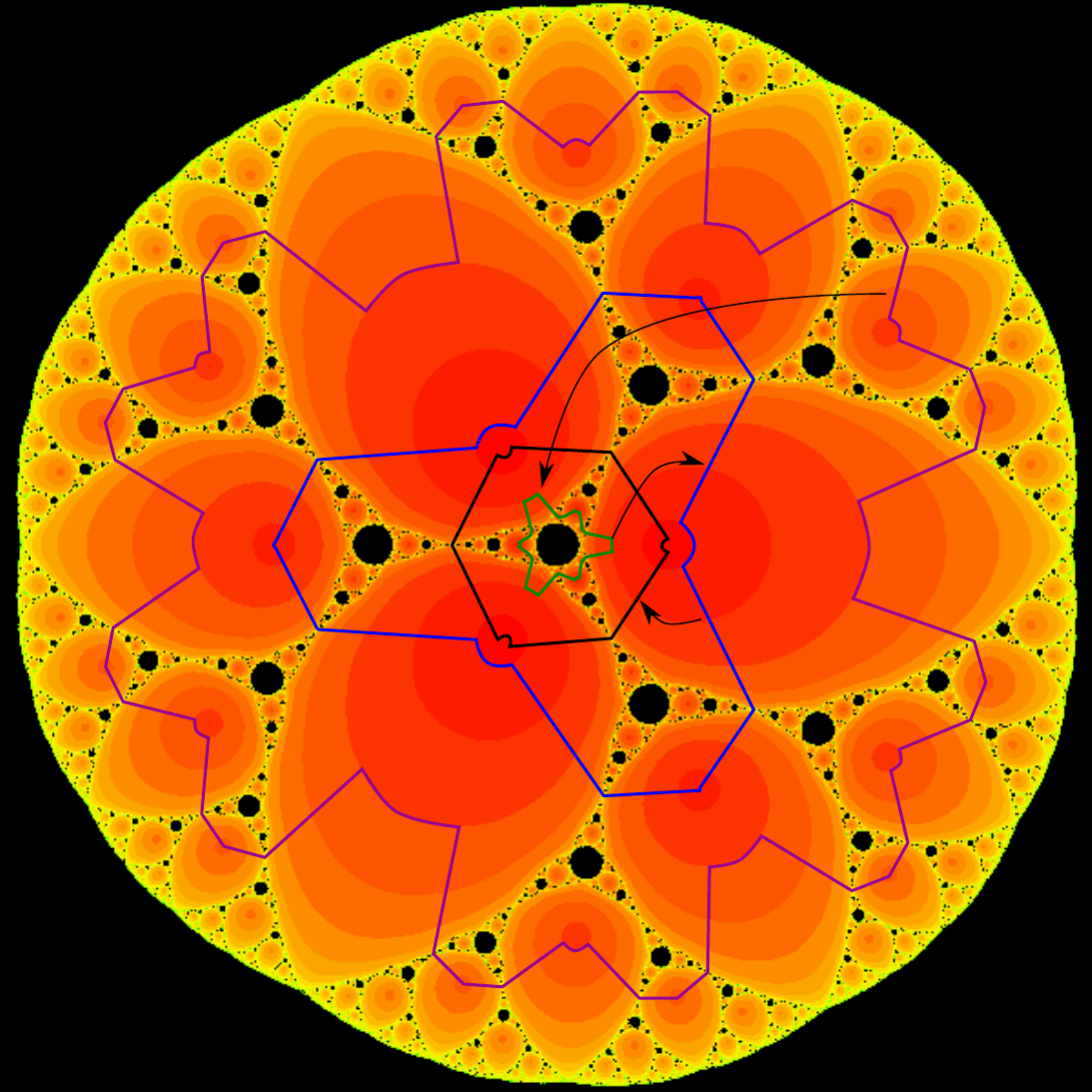}}%
   \put(0.48,0.385){\color[rgb]{0,0,0}\makebox(0,0)[lt]{\smash{\begin{tabular}[t]{l}$\gamma_0$\end{tabular}}}}%
    \put(0.40272002,0.37150467){\color[rgb]{0,0,1}\makebox(0,0)[lt]{\smash{\begin{tabular}[t]{l}$\gamma_1$\end{tabular}}}}%
    \put(0.45,0.435){\color[rgb]{0,0.54509804,0}\makebox(0,0)[lt]{\smash{\begin{tabular}[t]{l}$\gamma_2$\end{tabular}}}}%
    \put(0.33295981,0.22){\color[rgb]{0.58823529,0,0.58823529}\makebox(0,0)[lt]{\smash{\begin{tabular}[t]{l}$\gamma_3$\end{tabular}}}}%
  \end{picture}%
\caption{\small Configuration of the dynamics of the curve $\gamma_0$ and its preimages up to $\gamma_3$. the map $R_0$ maps each curve  $\gamma_k$  onto $\gamma_{k-1}$  with degree 2  for $k=1, \ldots, 4$.}
\label{fig:gamma}
\end{figure}
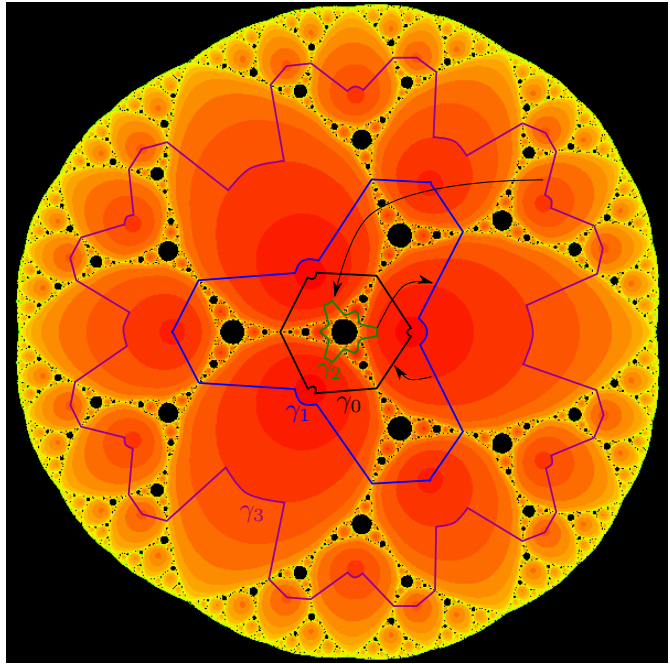

The curve  $\eta$ and its preimages  are well defined for any parameter $a$ close enough to $0$ (see next section). We want to use them to perform a cut and paste surgery (see \cite{BF}) to relate the dynamics of $R_a$ with the one of $z^4+\lambda/z^2$. However, these curves always intersect at the roots of the unity. To avoid this problem we now introduce a modified version of $\eta$ (see Figure~\ref{fig:gamma}). This modified version uses the equipotentials  \newr{given by} the Böttcher coordinates in the immediate basins of attraction of the third roots of the unity.

\begin{definition}\label{def:gamma0}
Fix $0<l_0<1$. We define $\gamma_0$ to be the simple closed curve obtained by cutting $\eta$ near the roots of the unity by the equipotentials of level $l_0$ and joining the cut points through these equipotentials  so that $Int(\gamma_0)$ is an open neighbourhood of $z=0$ not containing the roots of the unity.
\end{definition}

By construction $\gamma_0$ is invariant with respect to rotation by a third root of the unity. Notice that $\gamma_0$ depends on the choice of the level $l_0$. Notice also that the rotation with respect to a third root of the unity maps the equipotentials of level $l_0$ amongst themselves. We can now use the maps $R_0|_{Ext(\eta_{\infty})}$ and $R_0|_{Int(\eta_0)}$ (Lemma~\ref{lem:confD}) to take preimages of $\gamma_0$ (see Figure~\ref{fig:gamma}).

\begin{lemma}\label{lem:defpreimgamma}
Let $\gamma_1$ be the preimage of $\gamma_0$ contained in $\overline{ Ext(\eta_{\infty}) }$, $\gamma_{2}$ be the preimage of $\gamma_1$ contained in $Int(\eta_0)$,  $\gamma_{3}$ be the preimage of $\gamma_2$ contained in $Ext(\eta_{\infty}) $, and $\gamma_{4}$ be the preimage of $\gamma_3$ contained in $Int(\eta_0)$. Then the curves $\gamma_k$ for $k=1, \ldots, 4$ are simple closed curves which are invariant with respect to rotation by a third root of the unity, surround the origin,  and $R_0: \gamma_k \to \gamma_{k-1}$ has degree 2 for $k=1, \ldots, 4$.  Moreover, we have the inclusions $\gamma_4\subset Int(\gamma_2)$,  $\gamma_2\subset Int(\gamma_0)$, $\gamma_0\subset Int(\gamma_1)$, and $\gamma_1\subset Int(\gamma_{3})$.
\end{lemma}
\proof
These curves are well defined by the dynamics of $R_0|_{Ext(\eta_{\infty})}$ and $R_0|_{Int(\eta_0)}$ (see Lemma~\ref{lem:confD}). Their properties also come from the dynamics of $R_0|_{Ext(\eta_{\infty})}$ and $R_0|_{Int(\eta_0)}$. Notice that $\gamma_1$ coincides with $\eta_{\infty}$ except at the preimages of the subintervals replaced by equipotential segments.
\endproof

The next remark describes properties of  the curves $\gamma_1$, $\gamma_2$ and $\gamma_3$ which will be used in the next section. They follow from the dynamics of $R_0$ (Lemma~\ref{lem:confD}).

\begin{remark}\label{rem:preimgamma} All preimages of $\gamma_2$ under $R_0$, other than $\gamma_3$, are compactly contained in $Int(\gamma_3)$. Moreover, the open set $Int(\gamma_1)$ contains all preimages of $z=\infty$ under $R_0$.
\end{remark}

 We continue with a lemma showing that the Fatou components of $R_0$ are quasidisks.

\begin{lemma}\label{lem:boundary0}
 $\partial A_0^*(\infty)$, $\partial A_0^*(0)$ and $\partial A_0^*(\zeta^j)$ for $j=0,1,2$  are quasicircles.
\end{lemma}

\begin{proof}
We firstly consider the 2-cycle   $\{0,\infty\}$ of $R_0$. The triple $(R^2_0; Int(\gamma_2), Int(\gamma_0))$ is a polynomial-like map of degree 4 whose Julia set coincides with the boundary of $A_{0}^*(0)$. Notice that $z=0$ is a super-attracting fixed point of local degree 4 under $R^2$, so the polynomial-like map $(R^2_0, Int(\gamma_2), Int(\gamma_0))$ is hybrid equivalent to a polynomial of the form $b z^4$, $b\in\com$ and the result follows. 

We secondly consider the super-attracting fixed points located at $\zeta^j$ for $j=0,1,2$.  By symmetry, it is enough to prove the result for  $\partial A_0^*(1)$.  We construct a curve $\gamma$ which passes through the external rays of angles $1/4$ and $3/4$ in $\partial A_0^*(\zeta^1)$ and $\partial A_0^*(\zeta^2)$. We continue these curves in the immediate basins of attraction of $0$ and $\infty$ by following appropriate external rays and cutting by equipotentials so that it surrounds $z=1$  (see Figure~\ref{fig:pollyke1}). Moreover, $\gamma$ is modified to follow equipotentials near $\zeta^1$ and $\zeta^2$ so that it does not contain any of the fixed points (see Figure~\ref{fig:pollyke1}). Let $V$ be the domain bounded by $\gamma$ (which contains $z=1$). It is not difficult to show that, if the equipotentials in the basins of $z=0$ and $z=\infty$ are chosen appropriately, the connected component $U$ of $R_0^{-1}(V)$  which contains $\zeta^0=1$ is simply connected, is compactly contained in $V$ and is mapped with degree 3 onto $V$ under $R_0$. It follows that the triple  $(R_0; U, V)$ is a degree 3 polynomial-like mapping. Since $\zeta^0=1$ is a super-attracting fixed point of local degree 3, it follows that $R_0|_U$ is quasiconformally conjugate to $z^3$. This quasiconformal conjugacy maps the immediate basin of attraction of $z=1$ to the unit disk. Hence, we can conclude that $\partial A_0^*(1)$ is a quasicircle (see Figure~\ref{fig:pollyke1}).
\end{proof}

	\begin{figure}[hbt!]
		\centering
		\setlength{\unitlength}{250pt}%
	\begin{picture}(1,1)%
		\setlength\tabcolsep{0pt}%
		\put(0,0){\includegraphics[width=\unitlength,page=1]{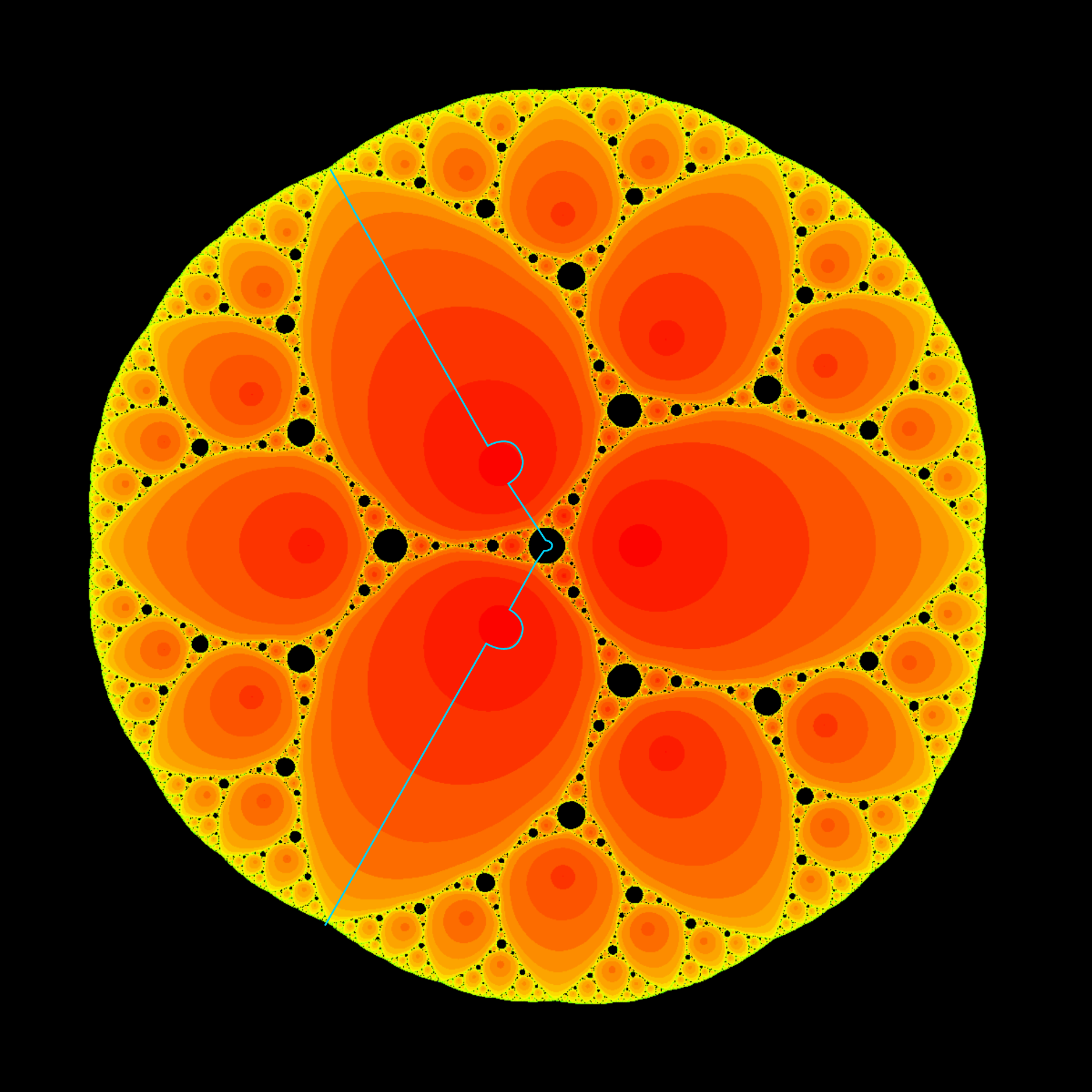}}%
		\put(0.37,0.72488636){\color[rgb]{0,0,0}\makebox(0,0)[lt]{\smash{\begin{tabular}[t]{l}$\frac{1}{4}$\\\end{tabular}}}}%
		\put(0.45,0.5178051){\color[rgb]{0,0,0}\makebox(0,0)[lt]{\smash{\begin{tabular}[t]{l}$\frac{3}{4}$\end{tabular}}}}%
		\put(0.4,0.3){\color[rgb]{0,0,0}\makebox(0,0)[lt]{\smash{\begin{tabular}[t]{l}$\gamma$\end{tabular}}}}%
		\put(0,0){\includegraphics[width=\unitlength,page=2]{ray2.pdf}}%
		\put(0.625,0.3){\color[rgb]{0,0,0}\makebox(0,0)[lt]{\smash{\begin{tabular}[t]{l}$\gamma^{-1}$\end{tabular}}}}%
		\put(0,0){\includegraphics[width=\unitlength,page=2]{ray2.pdf}}%
	\end{picture}%
\caption{\small Curves that delimit the polynomial-like mapping around $z=1$. }
\label{fig:pollyke1}
\end{figure}

We finish this section  by showing that the curve $\gamma_0$ is a quasicircle. Since $R_0$ has no \textit{free} critical points (all critical points of $R_0$ are fixed points), it follows that all preimages of \newr{$\gamma_0$ under} $R_0$ are also quasicircles.

\begin{corollary}\label{co:quasicircle0}
	The curve $\gamma_0$ is a quasicircle.
	\end{corollary}
\proof
The curve $\gamma_0$ is built by a finite union of analytic curves (dynamical rays of angles 0 and 1/2 and equipotentials   \newr{in} the basins of attraction of the third roots of the unity). In order to proof that $\gamma_0$ is a quasicircle it is enough to show that the analytic curves are joined forming positive angles. This is trivially true at the points where dynamical rays  \newr{connect} with equipotentials. The only problems could happen at the \newr{crossing}  of the dynamical rays at the repelling fixed points $x_{0,j}$. Since $\partial A^*_0(\zeta^j)$, $j = 0,1,2$, are quasicircles (Lemma~\ref{lem:boundary0}),  it follows that the dynamical rays land at the fixed points $x_{0,j}$, where $j = 0,1,2$, with a positive angle.

\endproof

\section{Dynamics of the map $R_a$}\label{sec:Dynamics:Ra}

In the next section we present a surgery construction relating the rational map   $R^2_a$ \eqref{eq:Ra} and the McMullen map  $M_{\lambda}$ \eqref{eq:Mcmullen}. This construction is based on the fact that the curves $\gamma_0$, $\gamma_1$, $\gamma_2$,  $\gamma_3$, and $\gamma_4$ (see Lemma~\ref{lem:defpreimgamma}) can be defined continuously for $|a|$ small enough keeping the same dynamics. We shall denote these continued curves by $\gamma_0(a)$, $\gamma_1(a)$, $\gamma_2(a)$, $\gamma_3(a)$, and $\gamma_4(a)$. 
The goal of the next lemmas is to introduce formally these curves. First, we introduce a lemma that allows us to locate the critical values which are the images of the critical points that appear for $a\neq 0$.

\begin{lemma}\label{lem:iterateanell}
	There exists a constant $\mathcal{C}_1>0$ that does not depend on $a$ such that if $|a|$ is small enough, $a\neq 0$, then $|R_a(z)|<\mathcal{C}_1|a|^{2/3}$ for all $z\in\mathbb{A}\left(\frac{1}{|a|^{1/3}}, \frac{3}{|a|^{1/3}}\right)$.
\end{lemma}
\begin{proof}
	
	Every point in $\mathbb{A}\left(\frac{1}{|a|^{1/3}}, \frac{3}{|a|^{1/3}}\right)$ can be written as $\frac{b}{a^{1/3}}$, where $1<|b|<3$.  We have

	$$R_a\left(\frac{b}{a^{1/3}}\right)=\frac{\frac{2b^6}{a}+\frac{15 b^3}{a}-b^3+3-a}{\frac{3b^2}{a^{2/3}}(5-a+\frac{b^3}{a}+b^3)}
	=\frac{(15b^3+2b^6)a^{2/3}+(3-b^3)a^{5/3}-a^{8/3}}{3b^5 +3b^2(5+b^3 -a)a}.$$
	Therefore, we have
	$$\left|R_a\left(\frac{b}{a^{1/3}}\right)\right|=\left|\frac{15+2b^3}{3b^2}\right||a|^{2/3} + o(|a|^{2/3}).$$
	To finish the proof it is enough to take 
	$$\mathcal{C}_1=1+\max_{1\leq|b|\leq 3}\left|\frac{15+2b^3}{3b^2}\right|.$$
\end{proof}

The curve $\gamma_0$ is defined using the dynamical rays of angles $0$ and $1/2$ of $\zeta^j$ cut through the equipotentials of level $\l_0$  so that $\zeta^j\notin Int(\gamma_0)$. The following lemma establishes the necessary conditions so that this construction can be repeated for $a\neq 0$ and serves as definition of $\gamma_0(a)$. 

\begin{lemma}\label{lem:contgamma0}
There exists $\Lambda_0\subset \com$  an open and simply connected set of parameters containing $a=0$ such that the following hold:
\begin{enumerate}[i)]
	\item The critical points $c_{a,j}$ (see  \eqref{eq:crit}) do not lie on the external rays of angles 0 or $1/2$ or the equipotentials of level $l_0$ of the immediate basins of attraction of $\zeta^j$, $j=0,1,2$.
	\item The fixed points $x_{a,j}$ (see \eqref{eq:fixed}) are repelling. 
\end{enumerate}
Moreover, for $a\in \Lambda_0$,  the curve $\gamma_0$ admits a holomorphic motion   whose image is a Jordan curve $\gamma_0(a)$  formed by   the dynamical rays of angles 0 and $1/2$ and the equipotentials of level $l_0$. This curve   $\gamma_0(a)$ is symmetric with respect to action  of $\mathbb U$ and  is a quasicircle  as $\gamma_0$. 
\end{lemma}
\proof
The fact that $\gamma_0(a)$ is a Jordan curve that is well defined by dynamical continuation of $\gamma_0$ follows directly from $i)$ and $ii)$. Since the curve $\gamma_0(a)$ is defined piecewise by dynamical objects that move holomorphically, it is a holomorphic motion of $\gamma_0$. Since $\gamma_0$ is a quasicircle (see Corollary~\ref{co:quasicircle0}) it follows from the $\lambda$-Lemma (see \cite{MSS}) that $\gamma_0(a)$ is also a quasicircle. 

Notice that the $x_{a,j}$ collide at $a=\pm3$ (see \eqref{eq:fixed}). However, those parameters do not belong to $\Lambda_0$. On the one hand, for $a=-3$ the fixed points $x_{-3,j}$ are parabolic. On the other hand,  $a=3$ is a singular parameter for which the degree of $R_a$ decreases to 4, the fixed points  $x_{3,j}$ and the critical points  $c_{3,j}$ collapse at $z=0$, and one of the fixed dynamical rays at the basin of attraction of each root of the unity lands at $z=\infty$ (the other one lands at $z=0$, which is a repelling fixed point for this singular parameter). This last claim is also satisfied for $|a-3|$ small enough,  so $a=3$ does not belong to $\partial \Lambda_0$. 
\endproof

\begin{remark}\label{rem:stabilityfixed}	The fixed points  $x_{a,j}$ are repelling for parameters in the complement of the closed disk of centre $a=-5$ and radius 2 (compare  \cite[Proposition 2.4]{CCV2}). 
\end{remark}

Once the curve $\gamma_0(a)$ is defined as a holomorphic motion of $\gamma_0$ over a set of parameters $\Lambda_0$ (Lemma~\ref{lem:contgamma0}), we can define recursively $\gamma_i(a)$ as a holomorphic motion of $\gamma_i$, $i=1,2,3,4$. In the next lemma we introduce these curves and describe their basic properties. 

\begin{lemma}\label{cont:gamma1234}
Define $\gamma_i(a)$, $i=1,2,3,4$,  recursively as follows. Let $\Lambda_{i}\subset \Lambda_{i-1}$ be an open simply connected set of parameters such that $\gamma_{i-1}(a)$ does not contain any critical value. Let $\gamma_i(a)$ to be the connected component of $R^{-1}_a(\gamma_{i-1}(a))$ which is a holomorphic motion of $\gamma_i(a)$. Then, the curves  $\gamma_i(a)$ satisfy the following properties:
\begin{enumerate}[i)]
	\item They are quasicircles and are symmetric with respect to rotation by a third root of the unity. 
	\item The curve $\gamma_i(a)$, $i=1,2,3,4$ is mapped 2 to 1 onto $\gamma_{i-1}(a)$ under $R_a$.
\end{enumerate}
Moreover, we have the inclusions $\gamma_4(a)\subset Int(\gamma_2(a))$,  $\gamma_2(a)\subset Int(\gamma_0(a))$, $\gamma_0(a)\subset Int(\gamma_1(a))$, and $\gamma_1(a)\subset Int(\gamma_{3}(a))$. 
\end{lemma}
\begin{proof}
	Since the sets $\Lambda_{i}$ are chosen so that the curves $\gamma_{i-1}(a)$ do not contain critical values and $\gamma_0(a)$ is a quasicircle (Lemma~\ref{lem:contgamma0}) it follows that all connected components of $R^{-1}_a(\gamma_{i-1}(a))$ are quasicircles. The curves $\gamma_i(a)$ are symmetric with respect to rotation by a third root of the unity since $\gamma_0(a)$ also is and this property is preserved by  \newr{backward iterations}   of $R_a$ (as long as the set surrounds $z=0$).
	
	The fact that the curves  $\gamma_i(a)$, $i=1,2,3,4$ are mapped 2 to 1 onto $\gamma_{i-1}(a)$ follows from the fact that they are holomorphic motions of $\gamma_i(a)$ and the curves $\gamma_i(a)$ satisfy the same property (see Lemma~\ref{lem:defpreimgamma}). The final inclusions also come from the corresponding inclusions of the curves  $\gamma_i(a)$.
\end{proof}

Notice that in the previous lemma we have defined recursively the sets $\Lambda_i$. We would like to point out that it was not strictly necessary to define all those sets due to the inclusions of the curves. Indeed, by Lemma~\ref{lem:iterateanell}, since $\gamma_0(a)\subset Int(\gamma_1(a))$ it follows that we can take $\Lambda_1=\Lambda_2$. Also, since $\gamma_2(a)\subset Int(\gamma_3(a))$ we can take $\Lambda_3=\Lambda_4$. Using the previous lemma we can now fix the set of parameters on which we will perform the surgery construction, which actually corresponds to $\Lambda_4$.

\begin{definition}\label{def:lambda} We define $\Lambda:=\Lambda_4$, i.e.\  $\Lambda\subset \com$ as an open simply connected set of parameters containing $a=0$ such that the holomorphic motion of $\gamma_2(a)$ is well defined and $\gamma_2(a)$ contains no critical value.
\end{definition}

\begin{figure}[hbt!]
	
	\centering
	\setlength{\unitlength}{230pt}%
	\begin{picture}(1,1)%
		\put(0,0){\includegraphics[width=\unitlength,page=1]{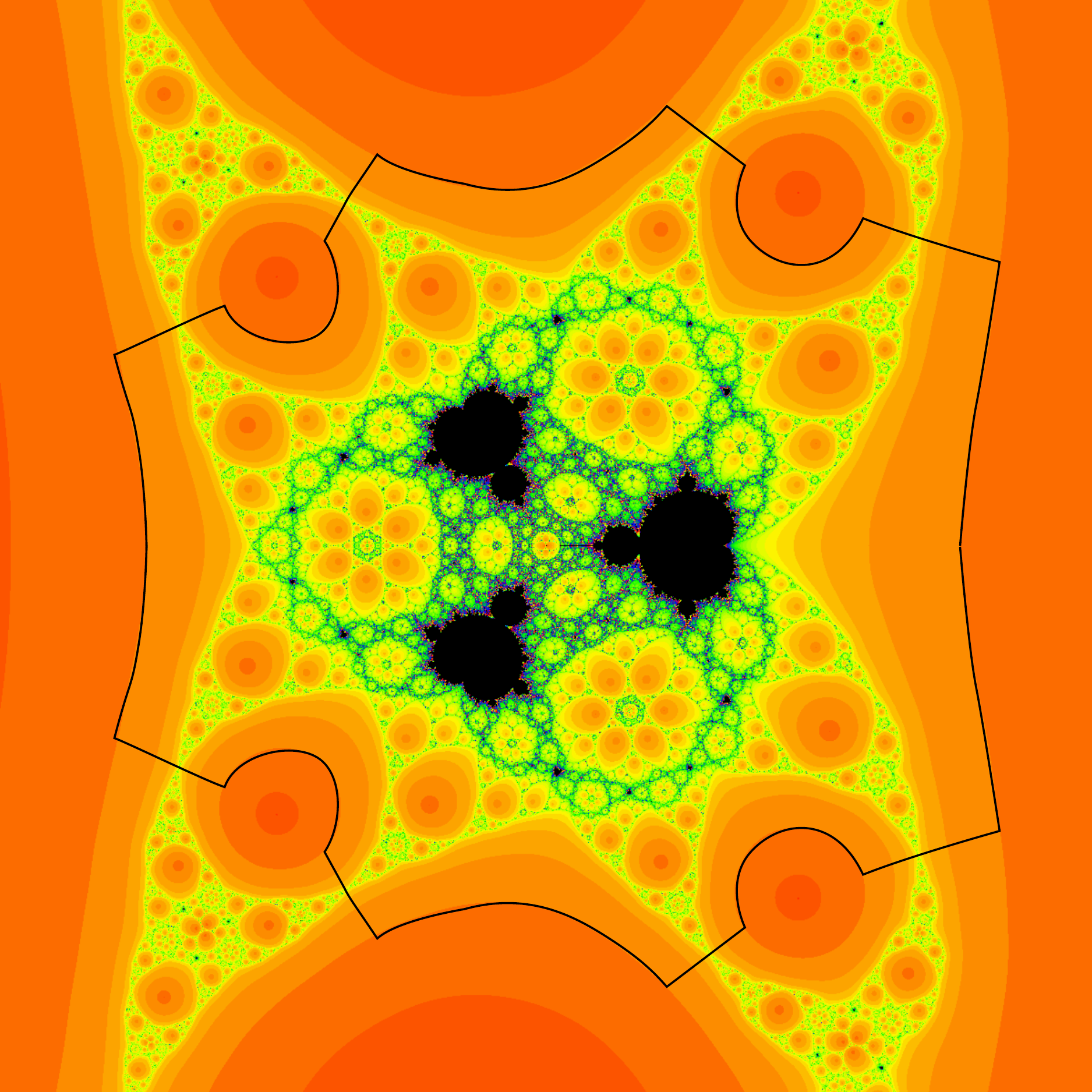}}%
		\put(0.8,0.5){\color[rgb]{0,0,0}\makebox(0,0)[lt]{\smash{\begin{tabular}[t]{l}$\Lambda$\end{tabular}}}}%
		\put(0.2,0.492){\color[rgb]{0,0,0}\makebox(0,0)[lt]{\smash{\begin{tabular}[t]{l}$ \bullet a_q$\end{tabular}}}}
		
	\end{picture}%
	\caption{\small Sketch of the domain $\Lambda$. The parameters $a$ are such that $\re(a)\in(-0.05, 0.05)$ and $\im(a)\in(-0.05, 0.05)$.}
	\label{fig:curveLambda}
\end{figure}

Even though it is not the goal of this paper, it follows from standard results in holomorphic dynamics that $\Lambda$ can be taken as indicated in Figure~\ref{fig:curveLambda}. The set $\partial\Lambda$ is chosen so that the critical values $v_{a,j}$ \eqref{eq:critvalue} lie \newr{on} $\gamma_2(a)$. Notice that the chosen $\Lambda$ is contained in the complement of the disk where the fixed points $x_{a,j}$ \eqref{eq:fixed} are non-repelling (see Remark~\ref{rem:stabilityfixed}).

We have proven that for  $a\in\Lambda$ the curves $\gamma_j(a)$ are mapped 2 to 1 onto $\gamma_{j-1}$ for $j=1,\ldots,4$. In this sense, the dynamics of $R_0$ is  \newr{preserved} after perturbation. The dynamics of $R_a$ restricted to the regions bounded by these curves is also maintained. However, after perturbation, the dynamics of $R_a$ on the unbounded regions delimited by these curves changes. This is described in the next proposition (see Figure~\ref{fig:confpert}). 

\begin{proposition}\label{prop:confpert}
Let $a$ in $ \Lambda \setminus \{0\}$. Then, there exists a preimage $\gamma_2'(a)$ of $\gamma_2(a)$ under $R_a$ which is a simple closed curve that is mapped 1 to 1 onto $\gamma_2(a)$,  is invariant with respect to rotation by a third root of the unity, and satisfies $\gamma_3(a)\subset Int(\gamma_2'(a))$. Moreover, the following holds.
\begin{enumerate}[i)]
\item The map $R_a:Int(\gamma_2(a))\rightarrow Ext(\gamma_1(a))$ is proper of degree 2.
\item The map $R_a:{A(\gamma_1(a),\gamma_3(a))}\rightarrow A(\gamma_2(a),\gamma_0(a))$ is proper of degree 2. 
\item The map $R_a:Ext(\gamma'_2(a))\rightarrow Ext(\gamma_{2}(a))$ is proper of degree 1.
\item The map $R_a:{A(\gamma_2'(a),\gamma_3(a))}\rightarrow Int(\gamma_2(a))$ is proper of degree 3. 
\end{enumerate}
In particular, the annulus $A(\gamma_2'(a),\gamma_3(a))$ contains the 3 critical points  and the 3 zeros that appear near $z=\infty$ after perturbation.
\end{proposition}

\begin{figure}[hbt!]
	\centering
	\setlength{\unitlength}{300pt}%
	\begin{picture}(1,0.78184249)%
		\put(0,0){\includegraphics[width=\unitlength,page=1]{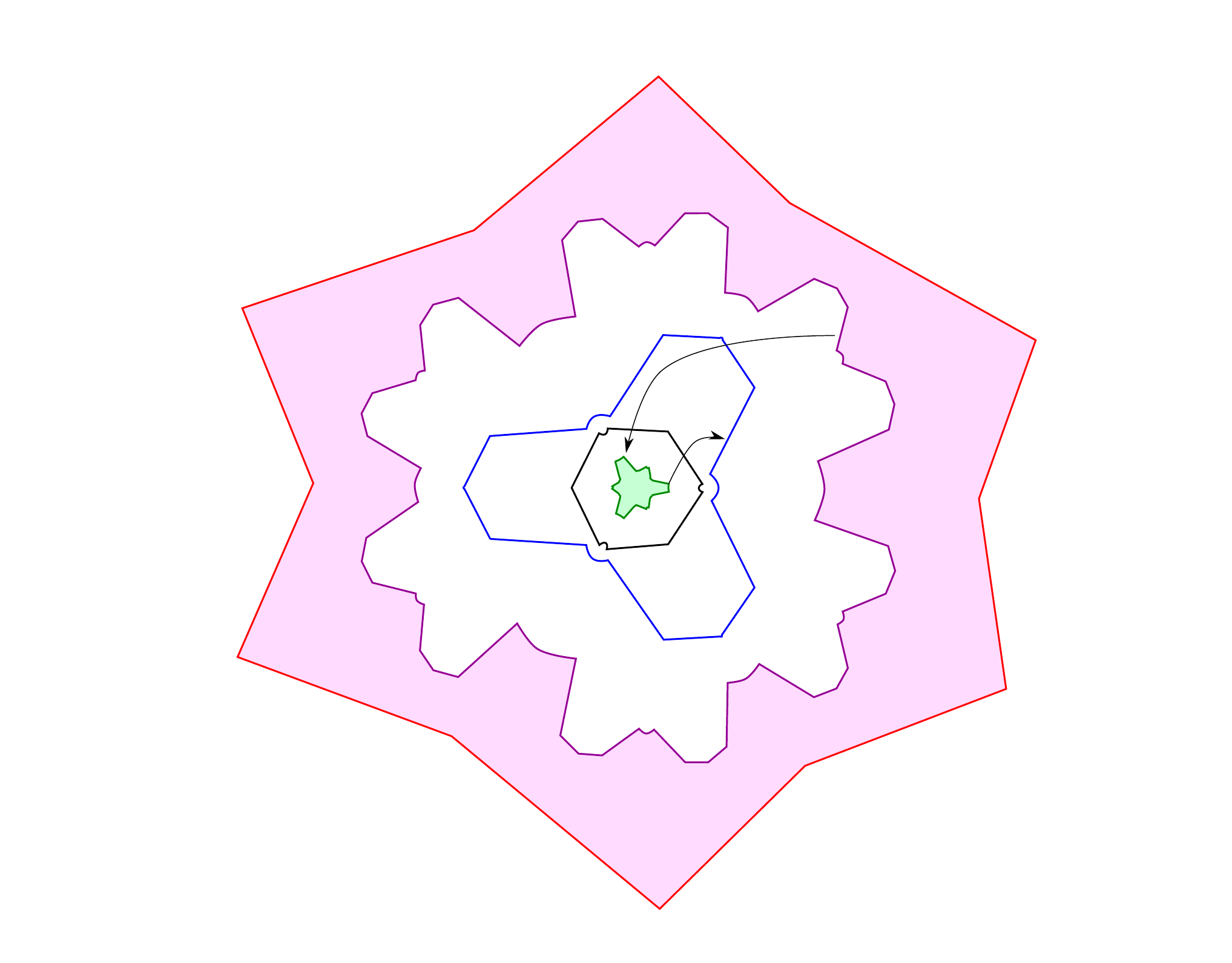}}%
		\put(0.5,0.315){\color[rgb]{0,0,0}\makebox(0,0)[lt]{\smash{\begin{tabular}[t]{l}$\gamma_0$\end{tabular}}}}%
		\put(0.53472778,0.24){\color[rgb]{0,0,1}\makebox(0,0)[lt]{\smash{\begin{tabular}[t]{l}$\gamma_1$\end{tabular}}}}%
		\put(0.495,0.35){\color[rgb]{0,0.54509804,0}\makebox(0,0)[lt]{\smash{\begin{tabular}[t]{l}$\gamma_2$\end{tabular}}}}%
		\put(0.58,0.5567109){\color[rgb]{0.58823529,0,0.58823529}\makebox(0,0)[lt]{\smash{\begin{tabular}[t]{l}$\gamma_3$\end{tabular}}}}%
		\put(0.48,0.66113087){\color[rgb]{1,0,0}\makebox(0,0)[lt]{\smash{\begin{tabular}[t]{l}$\gamma_2'$\end{tabular}}}}%
		\put(0,0){\includegraphics[width=\unitlength,page=2]{confpert.pdf}}%
		\put(0.37524422,0.01224631){\color[rgb]{0,0,0}\makebox(0,0)[lt]{\smash{\begin{tabular}[t]{l}$\gamma_0'$\end{tabular}}}}%
		\put(0,0){\includegraphics[width=\unitlength,page=3]{confpert.pdf}}%
		\put(0.36,0.4653136){\color[rgb]{0,0,0}\makebox(0,0)[lt]{\smash{\begin{tabular}[t]{l}$1:1$\end{tabular}}}}%
		\put(0.59,0.48){\color[rgb]{0,0,0}\makebox(0,0)[lt]{\smash{\begin{tabular}[t]{l}$2:1$\end{tabular}}}}%
		\put(0,0){\includegraphics[width=\unitlength,page=4]{confpert.pdf}}%
		\put(0.32,0.30526827){\color[rgb]{0,0,0}\makebox(0,0)[lt]{\smash{\begin{tabular}[t]{l}$3:1$\end{tabular}}}}%
			\put(0.7,0.37){\color[rgb]{0,0,0}\makebox(0,0)[lt]{\smash{\begin{tabular}[t]{l}$\times$\end{tabular}}}}%
		\put(0.38,0.55){\color[rgb]{0,0,0}\makebox(0,0)[lt]{\smash{\begin{tabular}[t]{l}$\times$\end{tabular}}}}%
		\put(0.38,0.21){\color[rgb]{0,0,0}\makebox(0,0)[lt]{\smash{\begin{tabular}[t]{l}$\times$\end{tabular}}}}%
		\put(0,0){\includegraphics[width=\unitlength,page=5]{confpert.pdf}}%
		\put(0.385,0.07387116){\color[rgb]{0,0,0}\makebox(0,0)[lt]{\smash{\begin{tabular}[t]{l}$1:1$\end{tabular}}}}%
	\end{picture}%
	\caption{\small Sketch of the dynamics of $R_a$ described in Proposition~\ref{prop:confpert}. The pink annulus $A(\gamma_2(a),\gamma_0(a)) $ contains three critical points (marked with crosses) and is mapped with degree 3 onto the green
disc bounded by $\gamma_2(a)$.  In order not to overload the figure, the dependence of the curves on $a$ is omitted. }
	\label{fig:confpert}
\end{figure}

\begin{proof}
Statement $i)$ follows from the fact that $a\in\Lambda$ \newr{since,} therefore, the dynamics in the region  bounded by $\gamma_2(a)$ remain unchanged. In particular,  the only pole of $R_a$ in $Int(\gamma_2(a))$ is $z=0$, which is a pole of order 2. Similarly, it is easy to see that $A(\gamma_1(a),\gamma_3(a))$ is a connected component of $R_a^{-1}(A(\gamma_2(a),\gamma_0(a))$ (notice that, by construction, $A(\gamma_1(a),\gamma_3(a))$ cannot contain neither zeros nor poles). Therefore, $R_a:{A(\gamma_1(a),\gamma_3(a))}\rightarrow A(\gamma_2(a),\gamma_0(a))$  is proper. By definition of $\Lambda$ the annulus $A(\gamma_2(a),\gamma_0(a))$  \newr{cannot contain} critical values. We conclude that the degree of the proper map is achieved on the boundaries of the annulus, and so this degree is 2. This \newr{proves} $ii)$.

For $a=0$, the curve $\gamma_3$ is mapped with degree 2 onto $\gamma_2$. Moreover, all \newr{the} other preimages of $\gamma_2$ lie in the region bounded by $\gamma_3$ (see Remark~\ref{rem:preimgamma}). Recall that $\Lambda$ consists of the open connected set of parameters containing $a=0$ such that no critical values has reached $\gamma_2(a)$. Equivalently, $\Lambda$ consists of the maximum set of parameters for which all preimages of $\gamma_2$ under $R_0$ can be continued as  preimages of $\gamma_2(a)$ under $R_a$. In particular, all preimages of $\gamma_2(a)$ in $\overline{Int(\gamma_3(a))}$ correspond to holomorphic motions of the preimages of $\gamma_2$ under $R_0$. Since $R_0$ has degree 5 and $R_a$ has degree 6, it follows that there is a simple closed curve $\gamma_2'(a)\subset Ext(\gamma_3(a)))$ which is mapped 1 to 1 onto $\gamma_2(a)$ under $R_a$. Moreover, $\gamma_2'(a)$ is invariant under rotation by a third root of the unity since  $\gamma_2(a)$ is also invariant. We obtain that $\gamma_3(a)\subset Int(\gamma_2'(a))$.

It follows from Remark~\ref{rem:preimgamma} that $Int(\gamma_1(a))$ contains all  \newr{the} preimages of $z=\infty$ other than itself. We can conclude that $R_a:Ext(\gamma'_2(a))\rightarrow Ext(\gamma_{2}(a))$ is proper of degree 1. This proves statement $iii)$. 

Since the curves $\gamma_2'(a)$ and $\gamma_3(a)$ are mapped onto $\gamma_2(a)$ with degree 1 and 2, respectively, and the annulus $A(\gamma_2'(a),\gamma_3(a))$ contains no preimage of $z=\infty$, it follows that $R_a:{A(\gamma_2'(a),\gamma_3(a))}\rightarrow Int(\gamma_2(a))$ is proper of degree 3. This proves statement $iv)$. The final claim follows from the Riemann-Hurwitz formula (see for instance \cite{Ste}) since 3 critical points, counting multiplicity, are required to map a doubly connected domain onto a simply connected domain via a proper map of degree 3. 
\end{proof}

\section{Surgery construction from $R_a^2$ to $z^4+\lambda/z^2$} \label{sec:surgery}

In this section we relate the dynamics of $R_a$ with the one of  $z^4+\lambda/z^2$. To do so we will perform a cut and paste surgery (see \cite{BF}). In this sense, the first step is to build a `rational-like map' which can be used to define the cut and paste surgery. This rational-like configuration (see Figure~\ref{fig:confsurgery}) is defined for $R^2_a$ and is based on the dynamics of $R_a$ described in Proposition~\ref{prop:confpert} \newr{(compare with Figure~\ref{fig:confpert} and Figure~\ref{fig:confpertbis})}. However, in order to make the main surgery construction easier to understand, \newr{since we are  looking at $R_a^2$}, we will introduce a new notation for some of the curves. 

\begin{proposition}\label{prop:rationallikeconf}
Let $a$ in $\Lambda \setminus \{0\}$. There exist quasicircles  $\beta_1^{in}(a)$, $\beta_2^{in}(a)$, $\beta_0^{out}(a)$, $\beta_1^{out}(a)$ and $\beta_2^{out}(a)$ which are analytic except on a finite set of points, surround $z=0$,  are invariant with respect to rotation by a third root of the unity, \newr{and} such that the following hold:
\begin{enumerate}[i)]
\item The curves $\beta_1^{out}(a)$ and $\beta_2^{out}(a)$ are mapped  with degree 4, under $R^2_a$, onto $\beta_0^{out}(a)$ and $\beta_1^{out}(a)$, respectively.
\item The curves $\beta_1^{in}(a)$ and $\beta_2^{in}(a)$ are mapped  with degree 2, under $R^2_a$, onto $\beta_0^{out}(a)$ and $\beta_1^{out}(a)$, respectively.
\item We have the inclusions 
\begin{itemize} 
\item	$\beta_1^{in}(a)\subset Int(\beta_2^{in}(a))$;
\item $\beta_2^{in}(a)\subset Int(\beta_2^{out}(a))$;
\item $\beta_2^{out}(a)\subset Int(\beta_1^{out}(a))$;
\item $\beta_1^{out}(a)\subset Int(\beta_0^{out}(a))$.
\end{itemize}
\item The map  $R_a^2$ \newr{satisfies}:
\begin{itemize}
	\item $R_a^2:A(\beta_2^{in}(a),\beta_2^{out}(a))\rightarrow Int( \beta_1^{out}(a))$ is proper of degree 6. 
	\item $R_a^2:A(\beta_1^{in}(a),\beta_1^{out}(a))\rightarrow Int( \beta_0^{out}(a))$ is proper of degree 6. 
\end{itemize} 
\end{enumerate}
\end{proposition}

\proof

We define $\beta_0^{out}(a):=\gamma_0(a)$, $\beta_1^{out}(a):=\gamma_2(a)$, and $\beta_2^{out}(a):=\gamma_4(a)$. Notice \newr{that}, by definition, $\beta_0^{out}(a)$ is a quasicircle which is analytic except  \newr{at} a finite set of points. This property is also satisfied by all its iterated preimages (as long as they do not contain a critical point).
Statement $i)$ follows directly from Lemma~\ref{cont:gamma1234}. 
By \newr{point $iii)$} of Proposition~\ref{prop:confpert} , there exists a simple closed curve $\gamma_0'(a)\subset Ext(\gamma_2'(a))$ that is mapped  with degree 1  onto $\beta_0^{out}(a)=\gamma_0(a)$, separates  $\gamma_2'(a)$ from $z=\infty$ and is symmetric with respect to rotation by a third root of the unity (see Figure~\ref{fig:confpert}).

\begin{figure}[hbt!]
	\centering
	\setlength{\unitlength}{350pt}%
	\begin{picture}(1,0.5)%

		\put(0,0){\includegraphics[width=\unitlength,page=1]{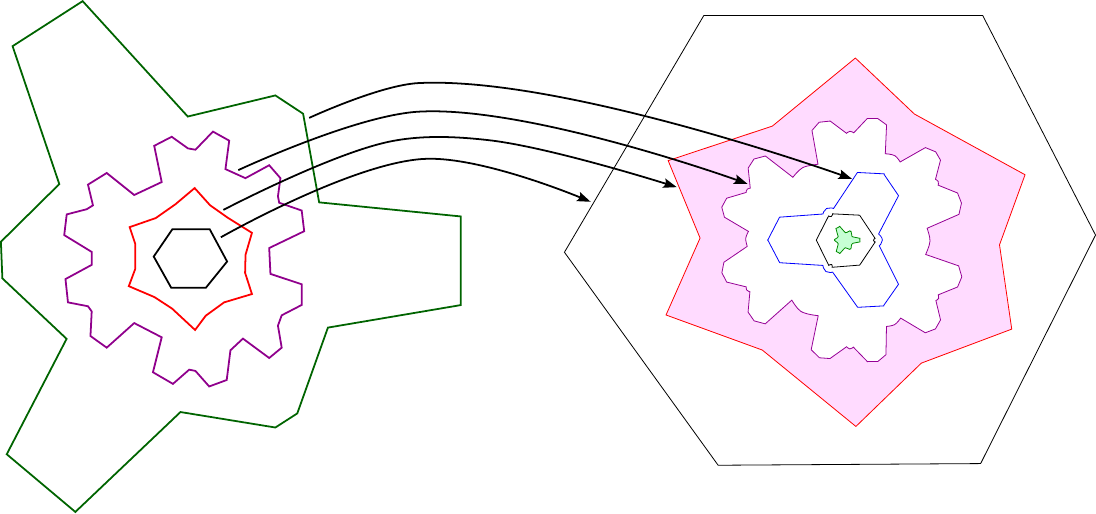}}%
		\put(0.62,0.42){\color[rgb]{0,0,0}\makebox(0,0)[lt]{\smash{\begin{tabular}[t]{l}$\gamma_0'$\end{tabular}}}}%
		\put(0.4,0.4){\color[rgb]{0,0,0}\makebox(0,0)[lt]{\smash{\begin{tabular}[t]{l}$R_a$\end{tabular}}}}%
		\put(0.83,0.37){\color[rgb]{1,0,0}\makebox(0,0)[lt]{\smash{\begin{tabular}[t]{l}$\gamma_2'$\end{tabular}}}}%
		\put(0.03,0.34){\color[rgb]{0.5,0,0.5}\makebox(0,0)[lt]{\smash{\begin{tabular}[t]{l}$\gamma_4=\beta_2^{out}$\end{tabular}}}}%
		\put(0.745,0.36){\color[rgb]{0.5,0,0.5}\makebox(0,0)[lt]{\smash{\begin{tabular}[t]{l}$\gamma_3$\end{tabular}}}}%
		\put(0.135,0.39){\color[rgb]{0,0.54509804,0}\makebox(0,0)[lt]{\smash{\begin{tabular}[t]{l}$\gamma_2=\beta_1^{out}$\end{tabular}}}}%
		\put(0.135,0.23){\color[rgb]{0,0,0}\makebox(0,0)[lt]{\smash{\begin{tabular}[t]{l}$\beta_1^{in}$\end{tabular}}}}%
		\put(0.19,0.17){\color[rgb]{1,0,0}\makebox(0,0)[lt]{\smash{\begin{tabular}[t]{l}$\beta_2^{in}$\end{tabular}}}}%
		\put(0.795,0.24){\color[rgb]{0,0,1}\makebox(0,0)[lt]{\smash{\begin{tabular}[t]{l}$\gamma_1$\end{tabular}}}}%
		
	\end{picture}%
	\caption{\small  \newr{Sketch of how the curves $\beta_1^{in}(a)$, $\beta_2^{in}(a)$, $\beta_2^{out}(a)$ and $\beta_1^{out}(a)$ are defined using the dynamics of $R_a$ described in Proposition~\ref{prop:confpert}. Each of the curves is mapped with degree 2 onto its image. The dependence of the curves on $a$ is omitted.}}
	\label{fig:confpertbis}
\end{figure}

The curves  $\beta_1^{in}(a)$ and  $\beta_2^{in}(a)$ are obtained by $i)$  of Proposition~\ref{prop:confpert}, taking \newr{the respective} preimage of   $\gamma_0'(a)$ and $\gamma_2'(a)$ contained in \newr{$Int(\gamma_2(a))$ (see Figure~\ref{fig:confpertbis})}. Recall here that, by $i)$ of Proposition~\ref{prop:confpert} $R_a:Int(\gamma_2(a))\rightarrow Ext(\gamma_1(a))$ is proper of degree 2, so $\beta_1^{in}(a)$ and $\beta_2^{in}(a)$ are mapped 2 to one onto $\beta_0^{out}(a)$ and $\beta_1^{out}(a)$, respectively, under $R^2_a$.
 Notice also that since $\gamma_3(a)$ separates $\gamma_2'(a)$ from $\gamma_1(a)$ and $\beta_2^{out}(a)=\gamma_4(a)$ is a preimage of $\gamma_3(a)$, we have that $\beta_1^{in}(a)\subset Int(\beta_2^{in}(a))$ and  $\beta_2^{in}(a)\subset Int(\beta_2^{out}(a))$. Together with Lemma~\ref{cont:gamma1234}, this finishes the proof of $i)$, $ii)$ and $iii)$.

Finally, we prove $iv)$. By $i)$ and $iii)$ of Proposition~\ref{prop:confpert}  we \newr{know that the maps} $R_a:Int(\gamma_2(a))\rightarrow Ext(\gamma_1(a))$ and  $R_a:{A(\gamma_2'(a),\gamma_3(a))}\rightarrow Int(\gamma_2(a))$ are proper  of degree 2 and 3, respectively. Recall that the curves $\beta_2^{in}(a)$ and $\beta_2^{out}(a)=\gamma_4(a)$ are the respective preimages of the curves $\gamma_2'(a)$ and $\gamma_3(a)$   in $Int(\gamma_2(a))=Int(\beta_1^{out}(a))$.  
So we can conclude that $R_a^2:A(\beta_2^{in}(a),\beta_2^{out}(a))\rightarrow Int( \beta_1^{out}(a))$ is proper of degree 6. This proper map can be extended to a degree 6 proper map $R_a^2:A(\beta_1^{in}(a),\beta_1^{out}(a))\rightarrow Int( \beta_0^{out}(a))$. This follows directly from  Proposition~\ref{prop:confpert}.
This finishes the proof.
%
%
\endproof

\begin{figure}[hbt!]
	\centering
	\setlength{\unitlength}{350pt}%
	\begin{picture}(1,0.44441143)%
		\setlength\tabcolsep{0pt}%
		\put(0,0){\includegraphics[width=\unitlength,page=1]{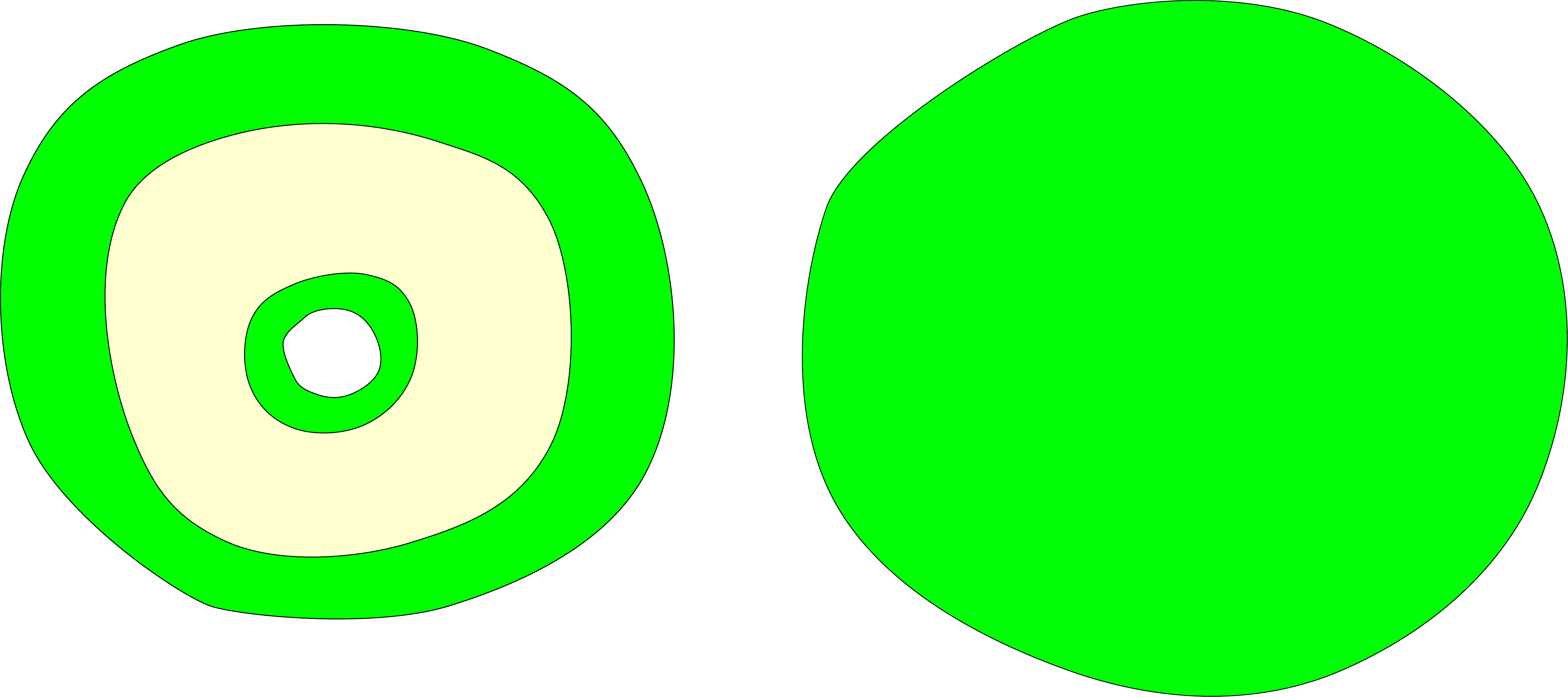}}%
		\put(0.19,0.22){\color[rgb]{0,0,0}\makebox(0,0)[lt]{\smash{\begin{tabular}[t]{l}$\beta_1^{in}$\end{tabular}}}}%
		\put(0.15,0.30){\color[rgb]{0,0,0}\makebox(0,0)[lt]{\smash{\begin{tabular}[t]{l}$\times$\end{tabular}}}}%
		\put(0.15,0.12){\color[rgb]{0,0,0}\makebox(0,0)[lt]{\smash{\begin{tabular}[t]{l}$\times$\end{tabular}}}}%
		\put(0.09,0.22){\color[rgb]{0,0,0}\makebox(0,0)[lt]{\smash{\begin{tabular}[t]{l}$\times$\end{tabular}}}}%
		\put(0.3,0.21){\color[rgb]{0,0,0}\makebox(0,0)[lt]{\smash{\begin{tabular}[t]{l}$\times$\end{tabular}}}}%
		\put(0.25,0.3){\color[rgb]{0,0,0}\makebox(0,0)[lt]{\smash{\begin{tabular}[t]{l}$\times$\end{tabular}}}}%
		\put(0.25,0.12){\color[rgb]{0,0,0}\makebox(0,0)[lt]{\smash{\begin{tabular}[t]{l}$\times$\end{tabular}}}}%
		
		\put(0.254,0.26){\color[rgb]{0,0,0}\makebox(0,0)[lt]{\smash{\begin{tabular}[t]{l}$\beta_2^{in}$\end{tabular}}}}%
		\put(0.28,0.15){\color[rgb]{0,0,0}\makebox(0,0)[lt]{\smash{\begin{tabular}[t]{l}$\beta_2^{out}$\end{tabular}}}}%
		\put(0.16430397,0.0575){\color[rgb]{0,0,0}\makebox(0,0)[lt]{\smash{\begin{tabular}[t]{l}$\beta_1^{out}$\end{tabular}}}}%
		\put(0.70407933,0.01449178){\color[rgb]{0,0,0}\makebox(0,0)[lt]{\smash{\begin{tabular}[t]{l}$\beta_0^{out}$\end{tabular}}}}%
		\put(0,0){\includegraphics[width=\unitlength,page=2]{confsurgery2.pdf}}%
		\put(0.69532559,0.07735862){\color[rgb]{0,0,0}\makebox(0,0)[lt]{\smash{\begin{tabular}[t]{l}$\beta_1^{out}$\end{tabular}}}}%
		\put(0,0){\includegraphics[width=\unitlength,page=3]{confsurgery2.pdf}}%
		\put(0.4165252,0.36){\color[rgb]{0,0,0}\makebox(0,0)[lt]{\smash{\begin{tabular}[t]{l}$R_a^2$\end{tabular}}}}%
		\put(0.43,0.317){\color[rgb]{0,0,0}\makebox(0,0)[lt]{\smash{\begin{tabular}[t]{l}$6:1$\end{tabular}}}}%
		
	\end{picture}%
	\caption{\small Sketch of the dynamics of the curves $\beta_1^{in}(a)$, $\beta_2^{out}(a)$, $\beta_2^{out}(a)$, and $\beta_1^{out}(a)$. The  yellow annulus   $A(\beta_2^{in}(a),\beta_2^{out}(a))$ is mapped under $R_a^2$ with degree 6 into the yellow  disc  $Int( \beta_1^{out}(a))$. We also mark the 6 critical points in the annulus with crosses. The dependence of the curves on $a$ is omitted. }
	\label{fig:confsurgery}
\end{figure}

\begin{remark}\label{rem:critR2}
	It follows from the previous statement that  $A(\beta_1^{in}(a),\beta_1^{out}(a))$ contains exactly 6 critical points. Indeed, by Proposition~\ref{prop:confpert} $i)$,  $A(\beta_2^{in}(a),\beta_2^{out}(a))$ is mapped 2 to 1 onto  $A(\gamma_2'(a),\gamma_3(a))$, which contains the 3 critical points $c_{a,j}$, $j=0,1,2,$ of $R_a$. We can conclude that if $a\in\Lambda$ the maps $R_a^2|_{A(\beta_1^{in}(a),\beta_1^{out}(a))}$ have exactly $6$ different critical points which are mapped under iteration of $R_a^2$ onto exactly 3 critical values. Notice that the 6 critical points cannot be mapped onto exactly one critical value since such critical value would have 12 preimages under $R_a^2|_{A(\beta_1^{in}(a),\beta_1^{out}(a))}$, counting multiplicity. This is impossible since $R_a^1|_{A(\beta_1^{in}(a),\beta_1^{out}(a))}$ is a degree 6 proper map.
\end{remark}

Once we have the rational-like configuration (Proposition~\ref{prop:rationallikeconf}), we can proceed to prove Theorem~A. This is the content of Theorem~\ref{thm:A}. This theorem relates the dynamics of the maps $R_a$ with the one of the McMullen maps $M_{\lambda}(z)=z^4+\lambda/z^2$ within the annulus $A(\beta_2^{in}(a),\beta_2^{out}(a))$.

\begin{theorem}\label{thm:A}
Let $a$ in $\Lambda \setminus \{0\}$. Then, there exists a quasiconformal map $\varphi_a:\wcom\rightarrow\wcom$ such that $\varphi_a\circ R^2_a(z)=M_{\lambda(a)}\circ \varphi_a(z)$ for all $z\in A(\beta_2^{in}(a),\beta_2^{out}(a))$, where $a \mapsto \lambda(a)$ is a map defined on $\Lambda$.  Moreover,  $\varphi_a^{-1}(\mathcal{J}(M_{\lambda(a)}))\subset A(\beta_2^{in}(a),\beta_2^{out}(a))$.
\end{theorem}

\proof
The idea of the proof is to perform a cut and paste surgery (see \cite{BF}). More specifically, we will build a model map that coincides with $R_a^2$ over $A(\beta_2^{in}(a),\beta_2^{out}(a))$, has the dynamics of $z^4$ and $1/z^2$ in $Ext(\beta_1^{out})$ and $Int(\beta_1^{in})$, respectively, and is globally quasisymmetric. Finally, we will use the Measurable Riemann Mapping Theorem (\cite{BF}) and Proposition~\ref{prop:rigidity} to conclude \newr{that} the model map is quasiconformally \newr{conjugated} to a map of the family $M_{\lambda}$.

We first explain how to glue the dynamics of $z^4$ in $Ext(\beta_1^{out})$ with the one of $R_a^2$ in $A(\beta_2^{in}(a),\beta_2^{out}(a))$. Pick $\rho >1$. Let

 $$\Phi_{out}:Ext(\beta_1^{out}(a))\rightarrow \hat{\com}\setminus \overline{\dis}_{\rho^4}$$ 
 
 \noindent be the Riemann map that has positive real derivative at $z=\infty$  and fixes it.  Since the curve $\beta_1^{out}(a)$ is invariant under rotation by a third root of the unity, it follows that $\Phi_{out}(\xi\cdot z)=\xi \cdot\Phi_{out}(z)$ for any third root of the unity $\xi$ and $z\in Ext(\beta_1^{out}(a))$.  Indeed, since the Riemann map fixing $z=\infty$ is unique up to rotation and $\xi^{-1}\Phi_{out}(\xi\cdot z)$ also has positive real derivative at $z=\infty$, it follows that $\xi^{-1}\Phi_{out}(\xi\cdot z)=\Phi_{out}(z).$
This property is also satisfied by the power map $\Phi_{out}^4(z)$.

Since $\beta_1^{out}(a)$ is a quasicircle, the map  $\Phi_{out}$ extends to the boundary as a quasisymmetric map (see \cite[Thm.\ 2.9]{BF}). Moreover, since  $\beta_1^{out}(a)$ is a finite union of analytic curves, this quasisymmetric maps is analytic except at a finite set of points (see \cite[Remark 2.12]{BF}). 
Let $\psi_{1,out}:\beta_1^{out}(a)\rightarrow\cercle_{\rho^4}$ be the extension map.
Since $\psi_{1,out}\circ R_a^2|_{\beta_2^{out}}(a):\beta_2^{out}(a)\rightarrow \cercle_{\rho^4}$ has degree $4$, we can choose a quasisymmetric lift $\psi_{2,out}:\beta_2^{out}(a)\rightarrow \cercle_{\rho}$ which is analytic except in a finite set of points so that $\psi_{1,out}(R_a^2(z))=\left(\psi_{2,out}(z)\right)^4$.
This lift $\psi_{2,out}$ can be chosen so that $\psi_{2,out}(\xi\cdot z)=\xi \cdot \psi_{2,out}(z)$ for any third root of the unity $\xi$.  By \cite[Proposition 2.30]{BF} there exists a quasiconformal map 

$$\psi_{out}:\overline{A}(\beta_2^{out}(a), \beta_1^{out}(a))\rightarrow \overline{\mathbb{A}}(\rho, \rho^4)$$  

\noindent  such that $\psi_{out}|_{\beta_2^{out}(a)}=\psi_{2 ,out}$ and that $\psi_{out}|_{\beta_1^{out}(a)}=\psi_{1 ,out}$. 
Moreover $\psi_{out}$ can be chosen so that  $\psi_{out}(\xi\cdot z)=\xi \cdot \psi_{out}(z)$ for any third root of the unity $\xi$. Indeed, \cite[Proposition 2.30]{BF} is based on \cite[Proposition 2.28]{BF}, which extends quasisymmetric boundary maps on a straight annulus to a quasiconformal map on the annulus, together with a uniformization map. It is not difficult to see that the quasiconformal map built in \cite[Proposition 2.28]{BF} is symmetry with respect to rotation by a third root of the unity if the boundary maps are also symmetric. As is the case with the Riemann map, the uniformization map sending a non straight annulus to a straight annulus can also be chosen to be symmetric.  

We can now define a quasiregular map in $Ext(\beta_2^{in}(a))$ as:

$$F_a(z)=\left\{\begin{array}{lcl}
	\Phi_{out}^{-1}\left(\Phi^4_{out}(z)\right)&  \mbox{ for } & z\in Ext(\beta_1^{out}(a))\\
	\Phi_{out}^{-1}\left(\psi^4_{out}(z)\right)&  \mbox{ for } & z\in \overline{A}(\beta_2^{out}(a), \beta_1^{out}(a))\\
	R_a^2(z) &  \mbox{ for } & z\in A(\beta_2^{in}(a), \beta_2^{out}(a)).
\end{array}\right.$$

To complete the model we have to glue the dynamics of $1/z^2$ in $Int(\beta_2^{in}(a))$. The construction is completely analogous to the previous case, so we skip some details.  
Let  
$$\Phi_{in}:Int(\beta_1^{in}(a))\rightarrow \dis_{1/\rho^4}$$ 
\noindent be the Riemann map that has positive real derivative at $z=0$  and fixes it. The map $\Phi_{in}$ is symmetric with respect to rotation by a third root of the unity.  Let $\psi_{1,in}:\beta_1^{in}(a)\rightarrow \cercle_{1/\rho^4}$ be the quasisymmetric extension of $\Phi_{in}$.
Since $\psi_{1,out}\circ R_a^2|_{\beta_2^{in}(a)}:\beta_2^{in}(a)\rightarrow \cercle_{\rho^4}$ has degree $2$, there exists a quasisymmetric lift $\psi_{2,in}:\beta_2^{in}(a)\rightarrow \cercle_{1/\rho^2}$ such that $\psi_{1,out}(R_a^2(z))=1/(\psi_{2,in}(z))^2$. 
By \cite[Proposition 2.30]{BF} there exists a quasiconformal map 
$$\psi_{in}:\overline{A}(\beta_1^{in}(a), \beta_2^{in}(a))\rightarrow \overline{\mathbb{A}}(1/\rho^4, 1/\rho^2)$$ 

 \noindent  such that $\psi_{in}|_{\beta_1^{in}(a)}=\psi_{1 ,in}$ and that $\psi_{in}|_{\beta_2^{in}(a)}=\psi_{2 ,in}$. 
As before, $\psi_{in}$ can be taken to be symmetric with respect to rotation by a third root of the unity. Finally, we can define our model map in the whole Riemann Sphere as:

$$F_a(z)=\left\{\begin{array}{lcl}
\Phi_{out}^{-1}\left(\Phi^4_{out}(z)\right)&  \mbox{ for } & z\in Ext(\beta_1^{out}(a))\\
\Phi_{out}^{-1}\left(\psi^4_{out}(z)\right)&  \mbox{ for } & z\in \overline{A}(\beta_2^{out}(a), \beta_1^{out}(a))\\
R_a^2(z) &  \mbox{ for } & z\in A(\beta_2^{in}(a), \beta_2^{out}(a))\\
\Phi_{out}^{-1}\left(\frac{1}{\psi^2_{in}(z)}\right)&  \mbox{ for } & z\in \overline{A}(\beta_1^{in}(a), \beta_2^{in}(a))\\
\Phi_{out}^{-1}\left(\frac{1}{\Phi^2_{in}(z)}\right)&  \mbox{ for } & z\in Int(\beta_1^{in}(a)),
\end{array}\right. $$

The map $F_a(z)$ is  quasiregular, is symmetric with respect to rotation by a third root of the unity, and has topological degree $6$ by construction. Moreover, $F_a$ is holomorphic in $\wcom\setminus \{A(\beta_1^{in}(a),\beta_2^{in}(a))\cup A(\beta_2^{out}(a), \beta_1^{out}(a))\}$. We continue by defining an $F_a$-invariant complex structure $\sigma$. Notice that the orbit of a point $z$ can go at most once through $A(\beta_1^{in}(a),\beta_2^{in}(a))\cup A(\beta_2^{out}(a), \beta_1^{out}(a))$. Denote $A_n=\{z\;|\; F_a^n(z)\in A(\beta_1^{in}(a),\beta_2^{in}(a))\cup A(\beta_2^{out}(a), \beta_1^{out}(a))\}$. Thus, it is enough to define

$$\sigma_a=\left\{\begin{array}{lcl}
\psi_{\infty}^{*}\sigma_0 &  \mbox{ for } & z\in  A(\beta_2^{out}(a), \beta_1^{out}(a))\\
\psi_{0}^{*}\sigma_0 &  \mbox{ for } & z\in  A(\beta_1^{in}(a),\beta_2^{in}(a))\\
(F_a^n)^{\circledast}\sigma_a &  \mbox{ for } & z\in A_n\\
\sigma_0&  & elsewhere,
\end{array}\right. $$

\noindent where $\sigma_0$ denotes the standard complex structure and $^*$ the pull-back operation.
By construction, $F_a^{*}\sigma_a=\sigma_a$. Since $F_a$ is  holomorphic outside $A(\beta_1^{in}(a),\beta_2^{in}(a))\cup A(\beta_2^{out}(a), \beta_1^{out}(a))$, $\sigma$ has bounded dilatation. Let $\xi$ denote any third root of the unity and let $O_{\xi}(z)=\xi\cdot z$. Since $F_a$ satisfies $O_{\xi}\circ F_a=F_a\circ O_{\xi}$, and so do $\psi_{out}$ and $\psi_{in}$, we have that $O_{\xi}^*\sigma_a=\sigma_a$. Let $\phi_a$ be the integrating map given by the Measurable Riemann Mapping Theorem (see  \cite[p.~57]{Ah} or \cite[Theorem 1.28]{BF}) which fixes $z=0$ and $z=\infty$ and is tangent to the  identity at $z=\infty$ (notice that $\phi_a$ is holomorphic in a neighbourhood of $z=\infty$). Then, $\phi_a^*\sigma_0=\sigma_a$. It follows from the unicity of the integrating map modulus post-composition with conformal automorphisms of $\wcom$ that $\phi_a=O_{\xi}^{-1}\circ\phi_a\circ O_{\xi}$ since $O_{\xi}^{-1}\circ\phi_a\circ O_{\xi}$ would satisfy the same normalizations and $(O_{\xi}^{-1}\circ\phi_a\circ O_{\xi})^*\sigma_0=O_{\xi}^*\sigma_a=\sigma_a$.

Finally, define $G_a=\phi_a \circ F_a\circ \phi^{-1}_a$. By construction, $G_a$ is a rational map of degree $6$. Given any third root of the unity $\xi$, the map $G_a$ satisfies $\xi\cdot G_a(z)=G_a(\xi\cdot z)$  since both $F_a$ and $\phi_a$ satisfy the same condition. By construction, $G_a$ maps $z=0$ to $z=\infty $ with local degree  2, the point $z=\infty$ is super-attracting of local degree $4$ and $G_a$ has 6 critical points which are mapped onto exactly 3 critical values (compare Remark~\ref{rem:critR2}). By Proposition~\ref{prop:rigidity} we conclude that $G_a$ is  conjugated to the map 
 $$M_{\lambda(a)}(z)=z^4+\frac{\lambda(a)}{z^2},$$
 
 \noindent under a linear map $L_a$. To finish the proof it is enough to take $\varphi_a=L_a\circ \phi_a$.  Notice that, by construction, $Ext(\beta_2^{out})$ and $Int(\beta_2^{int})$ belong to the basin of attraction of $z=\infty$ under $F_a$. Therefore, $\varphi_a^{-1}(\mathcal{J}(M_{\lambda(a)}))\subset A(\beta_2^{in}(a),\beta_2^{out}(a))$ 
\endproof

\begin{remark}
	The parameter $\lambda(a)$ depends in $a$ as well as in the level $l_0$ of the equipotentials chosen to define $\gamma_0$ (Definition~\ref{def:gamma0}). 
\end{remark}

\section{Further results}\label{sec:further_results}

In Theorem A we relate the dynamics of the map $R_a^2$ with the ones of a map $M_{\lambda(a)}$. 
The Escape Trichotomy Theorem (see section \ref{sec:MCM}) states that if all critical orbits of $M_{\lambda}$ escape to $\infty$, then $\mathcal{J}(M_{\lambda})$ is either a Cantor set, a Cantor set of circles or a Sierpinski carpet. In this section we study the Julia set of the maps $R_a$, justify that these three cases are achieved by the map  $M_{\lambda(a)}$ for different values of $a$ (compare Figure~\ref{fig:dyn_plane}). The next two results help us to understand how   the Julia set moves for $|a|$ small.

Let $\delta_0(0):=\partial A^*_0(0)$ and $\delta_\infty(0):=\partial A^*_0(\infty)$ be the boundaries of the immediate basin of attraction of 0 and $\infty$ under $R_0^2$. By Lemma~\ref{lem:boundary0}, $\delta_0(0)$ and  $\delta_\infty(0)$ are quasicircles. The next lemma states that there is a holomorphic motion of these curves in a small 
neighbourhood of $a=0$.

\begin{lemma}\label{lemma:motionboundaries}
There is a holomorphic motion $H(a, \cdot)$ of $\delta_0(0)\cup\delta_\infty(0)$ parameterized by a simply connected domain $\tilde{\Lambda}\subset \Lambda$ that is a neighbourhood of 0. In particular, for all $a\in\tilde\Lambda$ the curves $\delta_0(a)=H(a, \delta_0(0))$ and $\delta_\infty(a)=H(a, \delta_\infty(0))$ are quasicircles.
\end{lemma}
\proof The sets  $\delta_0(0)=\partial A^*_0(0)$ and $\delta_\infty(0)=\partial A^*_0(\infty)$  are Jordan curves. Periodic points are dense in those curves because they correspond to the rational angles in the Böttcher parametrization.  For $a=0$ there is no parabolic points so that all the aforementioned periodic points are repelling.  
We will prove that they stay  repelling  for  $|a|$   small enough.   Assuming  this property, we get a common neighbourhood of $a=0$  on which we can follow each repelling periodic point 
(by implicit function theorem). This defines  a holomorphic motion of the set of periodic point in the given curves.  Note that the neighbourhood can be chosen simply connected.
So,  
it then follows from the  $\lambda$-Lemma (see \cite{MSS}) that $\delta^0(0)$ and  $\delta^\infty(0)$  admit a holomorphic motion on this neighbourhood so that they are quasicircles through the motion.

We now prove the claim that there exists a neighbourhood of $a=0$ on which the periodic points of  $\partial A^*_0(0)$ stay repelling for  $|a|$   small enough (the proof is analogous for  $\partial A^*_0(\infty)$).
The idea is to perform a surgery which will eliminate all free critical points and keep the dynamics   of all periodic points coming from  $\partial A^*_0(0)$. Since the surgery construction is analogous to the classical one proposed by Douady and Hubbard \cite{DH1} for polynomial-like mappings, we only explain on which curves the cut and paste is done (see also \cite[Theorem 7.4]{BF}).  As in Theorem~\ref{thm:A}, we consider the map $R_a^2$. For $a=0$, the point $z=0$ is super-attracting of local degree 4. Let $\varsigma$ be a geodesic at $A_0^*(0)$, defined with the Böttcher coordinate.  Let $\varsigma_0^{-1}$ be the preimage of $\varsigma$ in $A_0^*(0)$ under $R_0^2$. Then $\varsigma_0^{-1}$ is also a geodesic and is mapped 4 to 1 onto $\varsigma$. 
 Since $R_a$ converges uniformly on compact sets of $\C$ to $R_0$, if $|a|$ is small enough then we can pick a connected component $\varsigma_a^{-1}$ of $R_a^{-2}(\varsigma)$ which is a continuation of $\varsigma_0^{-1}$.
Then, we can use $\varsigma_a^{-1}$ and $\varsigma$ and the curves $\beta^{out}_1(a)$ and $\beta^{out}_0(a)$ 
to perform a cut and paste surgery in which we glue the dynamics of $z^4$ near 0 and $\infty$ erasing all free critical points 
while keeping the dynamics of $R_a^2$ in the annulus bounded by $\beta^{out}_1(a)$ and $\varsigma_a^{-1}$. 
The resulting map is quasiconformally conjugate to $z^4$. Since the continuations of all periodic points of $\partial A^*_0(0)$
 and their orbits are contained in the annulus bounded by $\beta^{out}_1(a)$ and $\varsigma_a^{-1}$,
  this surgery keeps their dynamics. Moreover, since the resulting map does not have free critical points, we conclude that these periodic points are repelling.
\endproof

The next lemma tells us that the previous holomorphic motion can actually be extended to the \newr{closure} of the union of the basins of attraction of the roots under $R_0$,  $\overline {A_0(1)}\cup\overline {A_0(\zeta)}\cup\overline {A_0(\zeta^2)}$,  for all $a\in\tilde{\Lambda}$. This explains why for $|a|$ \newr{small} we can see copies of these basins of attraction on the dynamical plane of $R_a$ (see Figure~\ref{fig:dynampert}).

\begin{lemma}\label{lemma:motionbasins}
	There exists a  holomorphic motion of   $\overline {A_0(1)}\cup\overline {A_0(\zeta)}\cup\overline {A_0(\zeta^2)}$  which is parametrized by $\tilde{\Lambda}$  with  $H(a,\overline {A_0(\zeta^k)})\subset \overline{A_a(\zeta^k)}$ for $a\in\tilde{\Lambda}$ and $k=0,1,2$.
\end{lemma}
\proof Let $\phi_a$ be the Böttcher map of $R_a$ around the super-attracting fixed point $1$. It is well defined for $a\in\tilde\Lambda$  on  $A^*_a(1)$, the whole immediate basin of attraction  of $1$, since $A^*_a(1)$ is contained in the annulus $A_{\delta_a}$ bounded by $\delta_0(a)$ and $\delta_\infty(a)$, which   contains no free critical point (by Lemma~\ref{lemma:motionboundaries}). The map $H(a,z)=\phi_a(\phi_0^{-1}(z))$ is a holomorphic motion of the immediate basin of $1$, $A^*_0(1)$. It can be pulled back to every connected component of the basin of attraction of $1$  whose orbit never exits the annulus $A_{\delta_a}$ since $A_{\delta_a}$ does not contain any free critical point.
As a consequence, the holomorphic motion $H$ extends to the closure   $\overline {A_0(1)}$ and one easily sees that $H(a,\overline {A_0(1)})\subset\overline {A_a(1)}$. The argument is exactly the same  for $\overline {A_0(\zeta)}$ and $\overline {A_0(\zeta^2)}$.  \endproof

After  some  Lemmas    in order to  understand how  the Julia set moves with $a$, we show next that all cases of the Escape Trichotomy Theorem can be achieved and correspond to  maps $M_{\lambda(a)}$. First we introduce a technical lemma which will be useful to study the case of the Cantor set of quasicircles. Its proof is analogous to Lemma~\ref{lem:iterateanell}. 

\begin{lemma}\label{lem:iterate0}
	Let $\mathcal{C}>0$. There exists a $\mathcal{C}'>0$ such that if $|a|$ is small enough, $a\neq 0$, then $|R_a^2(z)|>\mathcal{C}'\frac{1}{|a|^{1/3}}$ for all $z$ such that $|z|<\mathcal{C}|a|^{2/3}$.
\end{lemma}

In the next proposition we study the case of Cantor set of quasicircles.
\begin{proposition}\label{prop:cantcircles}
	If $|a|$ is small enough, $a\neq 0$, then $\mathcal{J}(M_{\lambda(a)})$ is a Cantor set of quasicircles.
\end{proposition}
\proof

By Proposition~\ref{prop:rationallikeconf} $iv)$ and the Riemann-Hurwitz formula, we know that the annulus $A(\beta^{in}_2(a),\beta^{out}_2(a))$ contains 6 critical points and 6 zeros of $R_a^2$. Recall that, for $|a|$ small enough, the annulus $\mathbb{A}\left(\frac{1}{|a|^{1/3}}, \frac{3}{|a|^{1/3}}\right)$ contains the 3 free critical points of $R_a$ together with the 3 zeros that appear after the singular perturbation. Recall also that, by definition, $\beta^{out}_2(a)=\gamma_4(a)$, $\beta^{out}_1(a)=\gamma_2(a)$, and that $\beta^{out}_0(a)=\gamma_0(a)$.  

Furthermore,  there exists a connected component $\mathcal{A}_0(a)$ of $R_a^{-1}\left(\mathbb{A}\left(\frac{1}{|a|^{1/3}}, \frac{3}{|a|^{1/3}}\right)\right)$
 which is a doubly connected set  contained in $Int(\gamma_2(a))$ that is mapped with degree 2 onto $\mathbb{A}\left(\frac{1}{|a|^{1/3}}, \frac{3}{|a|^{1/3}}\right)$ under $R_a$, by Proposition~\ref{prop:confpert} $i)$. It follows that $\mathcal{A}_0(a)$ contains 6 critical points and 6 zeros of $R_a^2$, which correspond precisely to the 6 critical points and 6 zeros of $R_a^2$ in $A(\gamma_0''(a),\gamma_2(a))$ (notice that, by Proposition~\ref{prop:confpert}~$i)$, there is no other preimage of critical points of $R_a$ in $A(\beta^{in}_2(a),\beta^{out}_2(a))$). 

In application of Lemma~\ref{lem:iterateanell}, there exists a ${C}_1>0$ such that for $|a|$ small enough the set $R_a^2(\mathcal{A}_0(a))$ is contained in a disk of radius $\mathcal{C}_1|a|^{2/3}$. 
Since $A(\beta^{in}_1(a),\beta^{in}_2(a))$ is mapped onto $A(\beta^{out}_1(a),\beta^{out}_0(a))$ under $R_a^2$, it follows that $\mathcal{A}_0(a)\subset A(\beta^{in}(a),\delta_0(a)).$ Notice that, for $|a|$ small enough, the disk of radius $\mathcal{C}_1|a|^{2/3}$ is contained in the region bounded by $\delta_0(a)$. Moreover, by Lemma~\ref{lem:iterate0}, for $|a|$ small enough the set $\mathcal{A}_0$ is mapped under 2 iterates of $R_a^2$ onto $Ext(\gamma_0(a))$.

We can conclude that the critical points of $M_{\lambda(a)}$ are mapped under exactly 2 iterates of $M_{\lambda(a)}$ onto $A_{M_{\lambda(a)}}^*(\infty)$. It follows from the Escape Trichotomy that $\mathcal{J}(M_{\lambda(a)})$ is a Cantor set of quasicircles.

\endproof

Proposition~\ref{prop:cantcircles} gives us a condition so that the Julia set of $M_{\lambda(a)}$ is a Cantor set of quasicircles. We would like to understand  also the structure of the Julia set of the corresponding map $R_a$. The connectedness of the Julia set of the maps $O_{n,\alpha}$ obtained when applying Chebyshev-Halley methods to $z^n+c$, $c\in\com$ is studied in \cite{CCV2}. It is proven that these maps cannot have Herman rings. Moreover, the following characterization for the connectivity of $\mathcal{J}(O_{n,\alpha})$ is provided.  \newr{It}  depends on whether the immediate basin of attraction of 1, $\mathcal{A}^*_{O_{n,\alpha}}(1)$, contains extra critical points.

\begin{theorem}[\cite{CCV2}, Theorem 3.9]\label{thm:connectedjulia}
	For fixed $n\geq 2$ and $\alpha\in\com$, the Julia set $\mathcal{J}(O_{n,\alpha})$ is disconnected if and only if $\mathcal{A}^*_{O_{n,\alpha}}(1)$ contains a critical	point $c\neq 1$ and no preimage of $z = 1$ other than itself.
\end{theorem}

Notice that $R_a$ corresponds to $O_{3,\alpha}$ with $a=5-4\alpha$. It is not difficult to see that if $a$ is such that $\mathcal{J}(M_{\lambda(a)})$ is a Cantor set of quasicircles, then no free critical point of $R_a$ can belong to the immediate basin of attraction of $1$. It then follows from Theorem~\ref{thm:connectedjulia} that $\mathcal{J}(R_a)$ is connected. Moreover, from Proposition~\ref{prop:cantcircles} we know that $\mathcal{J}(R^2_a)$ contains  an invariant Cantor set of quasicircles (which separate 0 from $\infty$). The image under $R_a$ of this Cantor set of quasicircles is another Cantor set of quasicircles which also separate $0$ from $\infty$. 
From all the previous facts we obtain the next corollary.

\begin{corollary}
	If $a\in\Lambda$ and $\mathcal{J}(M_{\lambda(a)})$ is a Cantor set of quasicircles, then  $\mathcal{J}(R_a)$ is connected and contains an invariant Cantor set of quasicircles which separate 0 from $\infty$.
\end{corollary}

Next we study the case of the Cantor set of points. Recall that $\Lambda$ (see Definition~\ref{def:lambda}) is defined as an open simply connected set containing $a=0$ such that  $\gamma_2(a)$ can be continued and contains no critical value (and, hence, $\gamma_4(a)$ is well defined). Since the set of parameters for which the fixed points $x_{a,j}$  does not surround $a=0$ (see Remark \ref{rem:stabilityfixed}) we can choose  $\Lambda$ to contain parameters $a$ such that the critical values lie in $\gamma_4(a)$. It would follow directly that the critical values of $M_{\lambda(a)}$ lie in $A^*_{M_{\lambda(a)}}\newr{(\infty)}$ and, by the Escape Trichotomy,   $\mathcal{J}(M_{\lambda(a)})$ is a Cantor set of points. More specifically, we can prove the following. 

\begin{proposition}\label{prop:cantpoints}
	Let $a\in\Lambda\setminus\{0\}$. Let $A^0:=A(\gamma_2(a), \gamma_0(a))$ and let $A^1:=A(\gamma_4(a), \gamma_2(a))$. Define recursively $A^{n+1}$ as the connected component of $R^{-2}_a(A^n)$ which separates $0$ and $\infty$ and shares a boundary component with $A^n$. Then, $\mathcal{J}(M_{\lambda(a)})$ is a Cantor set of points if, and only if, the critical values of $R_a$ belong to $\overline{A^n}$ for some $n\geq1$. 
\end{proposition}
\proof
First we will mention why these sets are well defined. The fact that given $A^n$ there exists a connected component of $R^{-2}_a(A^n)$ satisfying \newr{the above conditions} follows inductively from the fact that $A^0$ and $A^1$ satisfy these conditions. Notice that $0$ cannot belong to any $A^n$ since it is mapped under  $R^{2}_a$  to $\infty$.

It is not difficult to see that the sets $A^n$ are sent to $A^*_{M_{\lambda(a)}}(\infty)$ under the surgery construction that defines $M_{\lambda(a)}$. Moreover, the critical values of $R_a$ coincide with the image under $R_a^2$ of the 6 critical points of $R_a^2$ which appear near 0 (and are preserved by the surgery construction\newr{)}. Therefore, if the critical values of $R_a$ lie in $\overline{A^n}$ for some $n\geq1$ we obtain that the critical values of $M_{\lambda(a)}$ lie in $A^*_{M_{\lambda(a)}}(\infty)$. By the Escape Trichotomy we can conclude that $\mathcal{J}(M_{\lambda(a)})$ is a Cantor set of points

Assume that there is no $n$ such that the critical values of $R_a$ belong to $\overline{A^n}$.  Then, by the Riemann-Hurwitz Formula, the sets $A^n$ are doubly connected. Moreover $\partial A^n\cap \partial A^{n+1}$ is a quasicircle and $A^{n+1}$ lies in the bounded component of $\com \setminus A^{n}$. Notice also that for all $n>1$ we have $A^n\subset A(\beta_2^{in}(a), \beta_2^{out}(a))$ (compare Proposition~\ref{prop:rationallikeconf}). In the limit, the sets $A^{n}$ need to accumulate on an invariant curve $\widehat{\delta}_a\subset A(\beta_2^{in}(a), \beta_2^{out}(a))$ which belongs to the Julia set of $R_a$. A quasiconformal copy of this curve will belong to $\mathcal{J}(M_{\lambda(a)})$. Therefore, $\mathcal{J}(M_{\lambda(a)})$ cannot be a Cantor set of points. 

We would like to point out that if $\widehat{\delta}_a$ contains no critical point, then it coincides with the $\delta_0(a)$ introduced in Lemma~\ref{lemma:motionboundaries}.

\endproof

Finally, we study the case of the Sierpinski carpet case. 

\begin{proposition}\label{prop:Sierpinski}
	There exists $a^*\in\Lambda\cap\mathbb{R^-}$ such that $\mathcal{J}(M_{\lambda(a^*)})$ is a Sierpinski carpet.  
\end{proposition}
\proof

It follows from the Escape Trichotomy Theorem that $\mathcal{J}(M_\lambda)$ is a Sierpinski carpet if, and only if, all critical points of $M_\lambda$ are mapped into $A^*_{M_\lambda}(\infty)$ in $m>2$ iterates.  Therefore, in order to prove the existence of a parameter $a^*$ for which $\mathcal{J}(M_\lambda(a^*))$ is a Sierpinski carpet it is enough to prove that the 6 critical points of $R_{a^*}^2$ which lie in $A(\beta^{in}_2(a^*),\beta^{out}_2(a^*))$ (see Proposition~\ref{prop:rationallikeconf} $iv)$, c.f.\ Proposition~\ref{prop:cantcircles}) are mapped in exactly 2 iterates of $R_{a^*}^2$ onto $z=0$ and thus they are mapped in 3 iterates onto $z=\infty$. It can easily be shown that in this case the origin is not contained in $A^*_{M_\lambda}(\infty)$.


Let $a\in\Lambda$. Let us denote the 6 critical points of $R_{a}^2$ which lie in $A(\beta^{in}_2(a),\beta^{out}_2(a))$ by $\tilde{c}_{a,k}$, $k=0,\cdots, 5$.
The critical points $\tilde{c}_{a,k}$ of $R_{a}^2$ are precisely the preimages in   $A(\beta^{in}_2(a),\beta^{out}_2(a))$ of the critical points $c_{a,j}, \, j=0,1, 2$ (see \eqref{eq:crit}). In particular, the images under $R_a^2$ of  $\tilde{c}_{a,k}$ coincides with the images under $R_a$ of $c_{a,j}$, that is, the critical values $v_{a,j}$ (see \eqref{eq:critvalue}). Therefore, we need to prove that there exists $a^*\in\Lambda$ such that $R^2_{a^*}(v_{a^*,j})=0$.

Hereafter  we restrict  to real parameters $a\in(-1,0)$.  As we mention before, the 3 free critical points of $R_a$ are denoted by $c_{a,j}$ (see  \eqref{eq:crit}) and the corresponding critical values by $v_{a,j}=R_a(c_{a,j})$ (see \eqref{eq:critvalue}). It is easy to check that $c_{a,0}$ is real and negative  and $v_{a,0}$ is real and positive for $ -1 < a < 0$.
%
%

Since \newr{every} attracting and parabolic cycle must contain a critical point in \newr{its} immediate basins of attraction, it follows that the real map $R_a$,  $a\in(-1,0)$, cannot have any attracting or parabolic cycle completely contained in $\mathbb{R^+}$. Indeed, if such a cycle exists, the critical point $c_{a,0}\in\mathbb{R}^-$ would belong to the immediate basin of attraction $\mathcal{A}^*(y)$ of an attracting or parabolic periodic point $y\in \mathbb{R}^+$. However, this is impossible since, by symmetry with respect to rotation by a third root of the unity, the critical points $c_{a,1}=\zeta c_{a,0}$ and $c_{a,2}=\zeta^2 c_{a,0}$ would belong to the immediate basins of attraction  $\mathcal{A}^*(\zeta y)$ and $\mathcal{A}^*(\zeta^2 y)$, where $\zeta=e^{2\pi i/3}$. Also by symmetry, the Fatou components $\mathcal{A}^*(y)$, $\mathcal{A}^*(\zeta y)$, and $\mathcal{A}^*(\zeta^2 y)$ would have non-empty intersection, which is impossible. 

If $a=0$ the map $R_0|_{\mathbb{R^+}}$ is strictly decreasing and satisfies $\lim_{x\rightarrow 0^+} R_0(x)=+\infty$ and  $\lim_{x\rightarrow +\infty} R_0(x)=0$ (notice that $x=1$ is super-attracting of local degree 3 and that there are no free critical points). It follows that the intersection of the immediate basin of attraction of 1, $A^*_{0}(1)$, with the real line consists of a period two cycle $\{q_0, q_{\infty}\}$ such that $0<q_0<1<q_\infty$. It is not difficult to see that $q_0$ is precisely the intersection of the curve  $\delta_0(0)=\partial A^*_0(0)$ (compare Lemma~\ref{lemma:motionboundaries}) with $\mathbb{R}^+$. After perturbation, for $|a|$ small, there is a holomorphic motion of the periodic point $q_0(a)$ (as well as the curve $\delta_0(a)$). Moreover, since for $a\in(-1,0)$ there can be no parabolic cycle completely contained in $\mathbb{R}^+$, the holomorphic motion of  $q_0(a)$ (and $q_\infty(a)$) is well defined for all $a\in (0,1)$ and we have $0<q_0(a)<1<q_\infty(a)$. Notice that when we move $a$ from $0$ up to $-1$ the critical value $v_{a,0}$ moves from $0$ up to $+\infty$. Therefore, we can define $a_q$ as the parameter in $(-1,0)$ such that  $0<v_{a,0}<q_0(a)$ if $a_q<a<0$ and $v_{a_q,0}=q_0(a_q)$.

The periodic point $q_0(a)$ and the parameter $a_q$ are important because they \newr{give} us a dynamical condition that we can control in order to ensure that a parameter $a\in(-1,0)$ is in the set $\Lambda$. 
Indeed,   before perturbation the point $q_0$ lies between $x=0$ and $\gamma_2\cap \R^+$ (compare Figure~\ref{fig:gamma}). The set of parameters $\Lambda$ is defined as an open simply connected set of parameters such that the holomorphic motion  $\gamma_2(a)$ is well defined and $\gamma_2(a)$ contains no critical value (see Definition~\ref{def:lambda}). Since the fixed points $x_{a,j}$ are repelling in the complement of the closed disk of centre -5 and radius 2 (see Remark~\ref{rem:stabilityfixed}), it follows that $\Lambda$ can be chosen to include a neighbourhood of the interval $(a_q, 0)$ (see Figure~\ref{fig:curveLambda}).

Now we can easily prove the existence of the parameter $a^*$. If $a\in(-1,0)$ the function $R_a|_{\R^+}$ is monotonous decreasing (it has no other critical point than $x=1$). When $x$ increases from 0 to 1,  $R_a(x)$ decreases from $+\infty$ down to 1.  
 When $x$ increases from 1 to $+\infty$,  $R_a(x)$ decreases from $1$ down to $-\infty$. Since $v_{a,0}$ tends to 0 when $a$ tends to 0, it is not difficult to show that $R_a(v_a,0)\rightarrow +\infty$ and  $R_a^2(v_a,0)\rightarrow -\infty$ when $a$ tends to 0 (compare with  Lemma~\ref{lem:iterate0} and proof of Proposition~\ref{prop:cantcircles}).
 Since for $a_q$ we have that  $v_{a_q,0}=q_0(a_q)$ and, hence, $R^2_{a_q}(v_{a_q,0})=q_0(a_q)>0$, we conclude that there is a parameter $a^*\in(a_q,0)\subset \Lambda$ such that $R^2_{a^*}(v_{a^*,0})=0$. This finishes the proof.

\endproof

Following the proof of Proposition~\ref{prop:Sierpinski}, in Figure~\ref{fig:dyn_plane} we show numerical examples of the three cases of the Escape Trichotomy with $a\in(-1,0)$. First we take $a$ negative and small enough ($a=-0.0003$) to show an example of a Cantor set of quasicircles. Then we take $a=-0.0164$, which is close to the $a^*$, to show the Sierpinski carpet case. Finally, we take $a=-0.028$, which is slightly smaller than the $a_q$, to show the Cantor set case (compare Figure~\ref{fig:curveLambda}). Notice that the parameter $a_q$ is precisely the limit until which the holomorphic motion $\delta_0(a)$ of the immediate basin of attraction of 0 for $a=0$ is well defined (see Lemma~\ref{lemma:motionboundaries}). Indeed, for $a_q$ the critical value $v_{a_q,0}$ coincides with the periodic point $y_0(a_q)$ (see proof of Proposition~\ref{prop:Sierpinski}). 

\subsection{Chebyshev-Halley methods applied to $z^n-1$}

\begin{figure}[t]
	\centering
	\subfigure{
		\includegraphics[width=170pt]{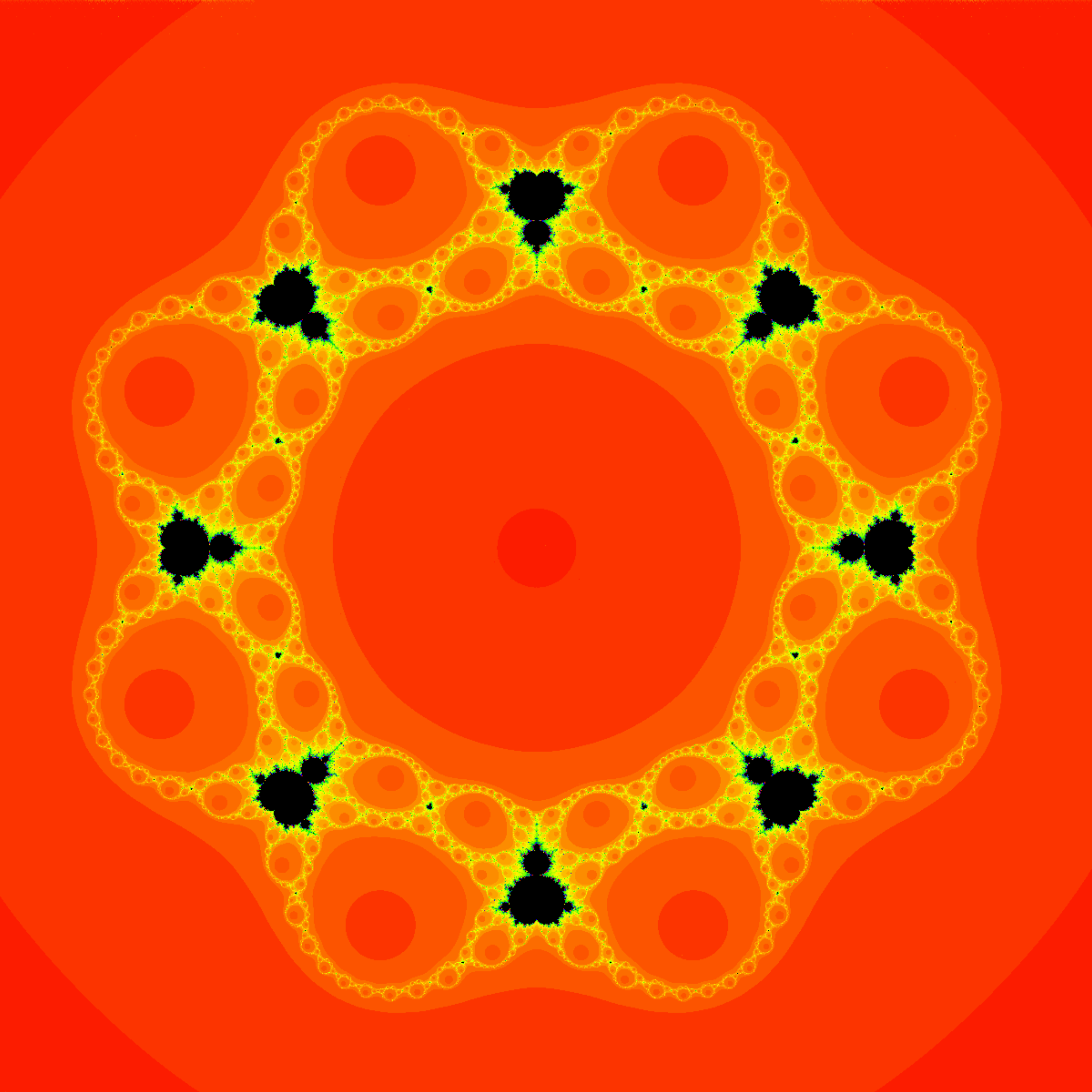}}
	\hspace{0.1in}
	\subfigure{
		\includegraphics[width=170pt]{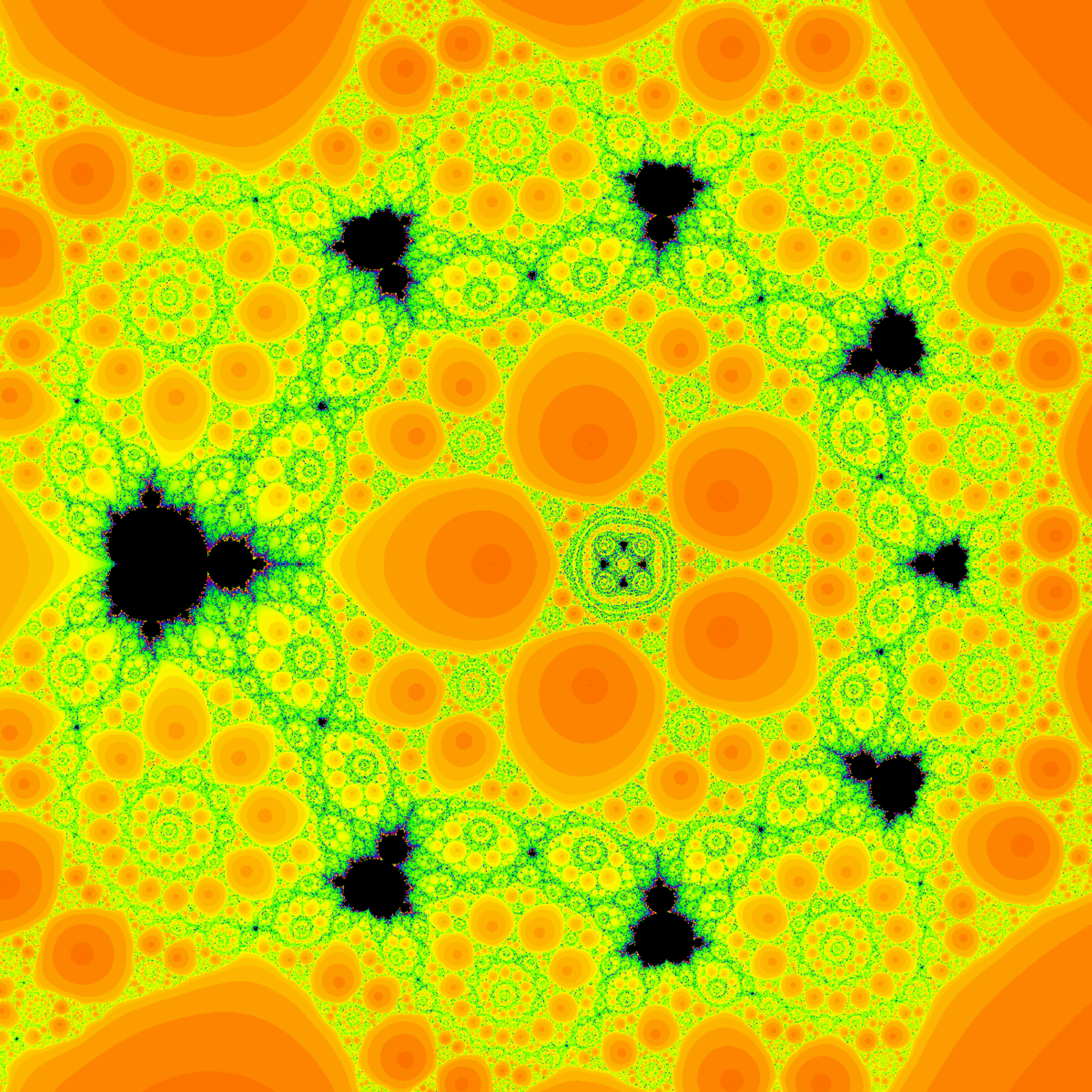}}
	\caption{\small{In the left side we show the parameter plane of  $M_{9,3,\lambda}$  and in the right side the parameter plane of  $O_{4,\alpha }$ near  $\alpha= 7/6$.} \label{fig:paramn4}}   
\end{figure}

We finish the paper with a remark. We have done the study of the Chebyshev-Halley methods applied to $z^3-1$. However, similar singular perturbations can be observed when these methods are applied to $z^n-1$ with $n\geq 3$. 
The operator obtained when applying Chebyshev-Halley methods to $z^n-1$ is given by the degree $2n$ rational map
{\small
\begin{equation*}
	O_{n,\alpha }(z)=z-\frac{(z^{n}-1)((-1+2\alpha +n-2\alpha n)+(1-2\alpha
		-3n+2\alpha n)z^{n})}{2nz^{n-1}(\alpha (n-1)(z^{n}-1)-nz^{n})}=
\end{equation*}
\begin{equation*}
	=\frac{(1-2\alpha )(n-1)+(2-4\alpha -4n+6\alpha n-2\alpha
		n^{2})z^{n}+(n-1)(1-2\alpha -2n+2\alpha n)z^{2n}}{2nz^{n-1}(\alpha
		(1-n)+(-\alpha -n+\alpha n)z^{n})},  \label{eq:Opn}
\end{equation*}
}

\noindent where $\alpha\in\mathbb{C}$. As in the degree 3 case,  these maps are symmetric with respect to $n$th roots of the unity and have a unique free critical orbit modulo symmetry (see \cite{CCV1, CCV2} for an introduction to the dynamics of these maps). The point $z=0$ is mapped onto $\infty$ with degree $n-1$ under $	O_{n,\alpha }$. If $\alpha\neq (2n-1)/(2n-2)$, then $z=\infty$ is a fixed point. However, if $\alpha= (2n-1)/(2n-2)$, then $z=\infty$ is mapped onto $z=0$ with degree $n-1$. As we have done for $n=3$, this can be studied from the point of view of singular perturbations. If $\alpha= (2n-1)/(2n-2)$, the point $z=0$ is a super-attracting fixed point of local degree $(n-1)^2$ of	$O^2_{n,\alpha}$. 
On the other hand, if $\alpha\neq (2n-1)/(2n-2)$ the point $z=0$ is mapped with degree $n-1$ onto $z=\infty$ under $O^2_{n,\alpha}$. It follows that, as we obtain for $n=3$, the dynamics near $z=0$ for parameters close to $\alpha= (2n-1)/(2n-2)$ can be related with the dynamics of the McMullen maps $M_{(n-1)^2, n-1,\lambda}(z)= z^{(n-1)^2}+\lambda/z^{n-1}$. In Figure~\ref{fig:paramn4} we show the parameter plane of  $O^2_{4,\alpha}$ near $\alpha= 7/6 $ and the parameter plane of the corresponding McMullen map $M_{9,3,\lambda}$.


\bibliography{bibliografia}

\def\cprime{$'$} \def\polhk#1{\setbox0=\hbox{#1}{\ooalign{\hidewidth
  \lower1.5ex\hbox{`}\hidewidth\crcr\unhbox0}}}
\begin{thebibliography}{10}

\bibitem{Ah}
L.~V. Ahlfors.
\newblock {\em Lectures on quasiconformal mappings}, volume~38 of {\em
  University Lecture Series}.
\newblock American Mathematical Society, Providence, RI, second edition, 2006.

\bibitem{BF}
B.~Branner and N.~Fagella.
\newblock {\em Quasiconformal surgery in holomorphic dynamics}, volume 141 of
  {\em Cambridge Studies in Advanced Mathematics}.
\newblock Cambridge University Press, 2014.

\bibitem{CCV1}
B.~Campos, J.~Canela, and P.~Vindel.
\newblock Convergence regions for the {C}hebyshev-{H}alley family.
\newblock {\em Commun Nonlinear Sci Numer Simulat}, 56:508--525, 2018.

\bibitem{CCV2}
B.~Campos, J.~Canela, and P.~Vindel.
\newblock Connectivity of the julia set for the chebyshev-halley family on
  degree n polynomials.
\newblock {\em Commun Nonlinear Sci Numer Simulat}, 82:105026, 2020.

\bibitem{CTV}
A.~Cordero, J.~R. Torregrosa, and P.~Vindel.
\newblock Dynamics of a family of {C}hebyshev-{H}alley type methods.
\newblock {\em Appl. Math. Comput.}, 219(16):8568--8583, 2013.

\bibitem{DLU}
R.~L. Devaney, D.~M. Look, and D.~Uminsky.
\newblock The escape trichotomy for singularly perturbed rational maps.
\newblock {\em Indiana Univ. Math. J.}, 54(6):1621--1634, 2005.

\bibitem{DH1}
A.~Douady and J.~H. Hubbard.
\newblock On the dynamics of polynomial-like mappings.
\newblock {\em Ann. Sci. \'Ecole Norm. Sup. (4)}, 18(2):287--343, 1985.

\bibitem{Ly3}
M.~Lyubich.
\newblock Feigenbaum-coullet-tresser universality and milnor's hairiness
  conjecture.
\newblock {\em Annals of Mathematics}, 149(2):319--420, 1999.

\bibitem{MSS}
R.~Ma{\~n}{\'e}, P.~Sad, and D.~Sullivan.
\newblock On the dynamics of rational maps.
\newblock {\em Ann. Sci. \'Ecole Norm. Sup. (4)}, 16(2):193--217, 1983.

\bibitem{McM1}
C.~McMullen.
\newblock Automorphisms of rational maps.
\newblock In {\em Holomorphic functions and moduli, {V}ol.\ {I} ({B}erkeley,
  {CA}, 1986)}, volume~10 of {\em Math. Sci. Res. Inst. Publ.}, pages 31--60.
  Springer, New York, 1988.

\bibitem{McMU}
Curtis~T. McMullen.
\newblock The {M}andelbrot set is universal.
\newblock In {\em The {M}andelbrot set, theme and variations}, volume 274 of
  {\em Lecture Notes in Math.}, pages 1--17. Cambridge Univ. Press, 2000.

\bibitem{Par}
D.~Paraschiv.
\newblock Newton-like components in the {C}hebyshev-{H}alley family of degree
  {$n$} polynomials.
\newblock {\em Mediterr. J. Math.}, 20(3):Paper No. 149, 17, 2023.

\bibitem{Ste}
N.~Steinmetz.
\newblock The formula of {R}iemann-{H}urwitz and iteration of rational
  functions.
\newblock {\em Complex Variables Theory Appl.}, 22(3-4):203--206, 1993.

\bibitem{traub1964}
J.F. Traub.
\newblock {\em Iterative methods for the solution of equations}.
\newblock Prentice-Hall, Englewood Cliffs, New Jersey, 1964.

\bibitem{Werner}
W.~Werner.
\newblock Some improvements of classical iterative methods for the solution of
  nonlinear equations.
\newblock In {\em Numerical solution of nonlinear equations ({B}remen, 1980)},
  volume 878 of {\em Lecture Notes in Math.}, pages 426--440. Springer,
  Berlin-New York, 1981.

\end{thebibliography}
\bibliographystyle{plain}

\end{document}